%% file: on_weak_limits_and_unimodular_measures.tex
\pdfoutput=1
\documentclass[12pt]{article}
\usepackage[letterpaper,left=3.5cm,right=2.5cm,top=2.5cm,bottom=2.5cm]{geometry}
\usepackage{amsmath}
\usepackage{amsthm}
\usepackage{amssymb}
\usepackage{graphicx}
\usepackage{empheq}
\usepackage{framed}
\usepackage{makeidx}
\usepackage[bookmarks=false,linktocpage=true]{hyperref}
\makeindex

\newcommand{\NN}{\mathbb{N}}
\newcommand{\ZZ}{\mathbb{Z}}
\newcommand{\RR}{\mathbb{R}}

\newcommand{\ol}[1]{\overline{#1}}

\newcommand{\wt}[1]{\widetilde{#1}}

\newcommand{\M}{\mathfrak{M}}
\newcommand{\Aut}{\mathrm{Aut}}
\newcommand{\U}{\mathcal{U}}
\newcommand{\Gr}{\mathbf{Gr}}
\newcommand{\dGr}{\vec{\Gr}}
\newcommand{\C}{\mathbf{C}}
\newcommand{\N}{\mathcal{N}}
\newcommand{\Rcc}{\mathrm{Rcc}}
\newcommand{\BRcc}{\mathrm{BRcc}}
\newcommand{\Cay}{\mathrm{Cay}}

\newcommand{\e}{\varepsilon}
\newcommand{\ceil}[1]{\left\lceil #1 \right\rceil}

\newtheorem{theo}{Theorem}[section]
\newenvironment{introthm}[1]{\renewcommand\thetheo{#1}\begin{theo}}{\end{theo}}
\newtheorem{lem}[theo]{Lemma}
\newtheorem{prop}[theo]{Proposition}
\newtheorem{cor}[theo]{Corollary}
\theoremstyle{definition}
\newtheorem{defn}[theo]{Definition}
\newtheorem{eg}[theo]{Example}
\newtheorem{rem}[theo]{Remark}

\begin{document}
\renewcommand{\thepage}{\roman{page}}
\thispagestyle{empty}

\begin{center}
\LARGE\upshape{On Weak Limits and Unimodular Measures}
\end{center}

\vfill

\begin{center}
\large\rm Igor Artemenko
\end{center}

\vfill

\begin{center}
Thesis submitted to the Faculty of Graduate and Postdoctoral Studies in partial fulfillment of the requirements for the degree of Master of Science in Mathematics\footnote{The M.Sc. program is a joint program with Carleton University, administered by the Ottawa-Carleton Institute of Mathematics and Statistics}
\end{center}

\vfill

\begin{center}
Department of Mathematics and Statistics\\
Faculty of Science\\
University of Ottawa
\end{center}

\vfill

\begin{center}
\copyright\ Igor Artemenko, Ottawa, Canada, 2014
\end{center}
\clearpage

%
%
\input{abstract}
\cleardoublepage

%
%
\input{acknowledgements}
\cleardoublepage

%
%
\tableofcontents
\clearpage

%
%
\phantomsection
\addcontentsline{toc}{section}{List of Figures}
\listoffigures
\clearpage

%
%
\setcounter{page}{1}
\renewcommand{\thepage}{\arabic{page}}

\phantomsection
\addcontentsline{toc}{section}{Introduction}
\input{introduction}
\cleardoublepage

\input{preliminaries}
\cleardoublepage

\input{uniqueness_and_existence_of_sustained_unimodular_measures}
\cleardoublepage

\input{judiciality_of_disconnected_graphs}
\cleardoublepage

\input{weak_limits_are_invariant_under_negligence}
\cleardoublepage

\input{counterexamples_and_open_problems}
\cleardoublepage

%
%
\phantomsection
\addcontentsline{toc}{section}{References}
\nocite{*}
\bibliographystyle{alpha}
\bibliography{on_weak_limits_and_unimodular_measures}
\cleardoublepage

%
%
\phantomsection
\addcontentsline{toc}{section}{Index}
\printindex
\end{document}

%% file: abstract.tex
%
%
\section*{Abstract}

In this thesis, the main objects of study are probability measures on the isomorphism classes of countable, connected rooted graphs. An important class of such measures is formed by unimodular measures, which satisfy a certain equation, sometimes referred to as the intrinsic mass transport principle. The so-called law of a finite graph is an example of a unimodular measure. We say that a measure is sustained by a countable graph if the set of rooted connected components of the graph has full measure. We demonstrate several new results involving sustained unimodular measures, and provide thorough arguments for known ones.

In particular, we give a criterion for unimodularity on connected graphs, deduce that connected graphs sustain at most one unimodular measure, and prove that unimodular measures sustained by disconnected graphs are convex combinations. Furthermore, we discuss weak limits of laws of finite graphs, and construct counterexamples to seemingly reasonable conjectures.

%% file: acknowledgements.tex
%
%
\section*{Acknowledgements}

I am grateful to my supervisor Dr. Vladimir Pestov with whom I met frequently, and whose constant advice shaped my research into the thesis you see before you. Dr. Pestov encouraged my interest in unimodularity, and his guidance during my undergraduate work in the Summers of 2010 and 2011 gave me an excellent base to work with. Our conversations were enjoyable and usually led to new insights and ideas. Thank you, Professor.

I would like to express my sincere gratitude for the NSERC CGS M (Alexander Graham Bell Canada Graduate Scholarship) and the OGS (Ontario Graduate Scholarship), which provided financial support during the course of my Master's degree. In addition, I am thankful for the internal scholarships given by the Department of Mathematics and Statistics and the Faculty of Graduate and Postdoctoral Studies at the University of Ottawa, which alleviated my worries about tuition fees.

Furthermore, I give thanks to my wonderful parents, Alexandr and Lyudmila, who supported me throughout my entire life, including the final stretch of this thesis. Without them, I would not have had nearly as much success in life. Thank you for giving me the wonderful opportunities I have today, and for passing on your remarkable work ethic.

Most importantly, I am extremely fortunate to have found the love of my life while finishing my graduate degree. My incredible girlfriend \'Emilie made sure I did not dwell on false hopes during my research, and kept me moving forward. She tried her best to keep me happy, and reminded me to relax once in a while. I do not think I could have succeeded were it not for her.

%% file: introduction.tex
%
%
\section*{Introduction}

The topic of unimodularity is discussed in many papers, and it is defined in several ways. It was introduced by Benjamini and Schramm for the purpose of applying the mass transport principle in a graph theoretic setting \cite{benjaminischramm01}, and Aldous and Steele, who used it for deeper applications of the objective method \cite{aldoussteele03}. Since then, it has been reformulated to suit the needs of the author. Prominent authors in this area include G\'abor Elek, Russell Lyons, and David Aldous whose works have provided much of the inspiration for this thesis \cite{aldouslyons07,elek07,elek10}. However, our approach is different, and the conclusions were usually discovered independently. In addition to new results, this thesis includes detailed arguments of known results, as well as generalizations of this author's previous statements \cite{artemenko11a,artemenko11b}.

The background knowledge for the majority of this thesis is rigorously presented in the previous works of this author \cite{artemenko11a,artemenko11b}. Nevertheless, the ideas are carefully discussed to ensure that the majority of this thesis is self-contained.

The collections $\Gr$ of rooted graphs, and $\dGr$ of birooted graphs are compact ultrametric spaces. Every graph $X$ induces a subset $\Rcc(X) \subseteq \Gr$ of rooted connected components $[X_x,x]$ for all $x \in V(X)$. A probability measure $\mu$ on $\Gr$ is sustained by $X$ if $\Rcc(X)$ has full measure; $\mu$ is strictly sustained if each rooted connected component has a positive measure. In the context of measures sustained by countable graphs, a probability measure $\mu$ sustained by a graph $X$ is unimodular if and only if
\[
\sum_{[X_x,x] \in \Rcc(X)} \sum_{y \in N(x)} f[X_x,x,y] \cdot \mu[X_x,x] = \sum_{[X_x,x] \in \Rcc(X)} \sum_{y \in N(x)} f[X_x,y,x] \cdot \mu[X_x,x]
\]
for all nonnegative functions $f$ on $\dGr$. As an extension of the result shown in this author's Honours project \cite{artemenko11a}, we demonstrate that a unimodular measure sustained by a connected graph is strictly sustained.

Graphs that sustain unimodular measures are said to be judicial; those that do not are lawless. Every finite graph $X$ is judicial because it sustains the probability measure $\Psi(X)$ known as the law of $X$:
\[
\Psi(X)[X_x,x] = \frac{|\Aut(X)x|}{|V(X)|}
\]
for all $[X_x,x] \in \Rcc(X)$, and $\Psi(X) = 0$ otherwise \cite{schramm08}.

We discuss and demonstrate the following new observations. Infinite connected graphs whose orbits are all finite are lawless. In particular, the same is true for infinite connected rigid graphs. On the other hand, connected graphs sustain at most one unimodular measure.

\begin{introthm}{\ref{uniqueness_of_sustained_unimodular_measures}}
Let $X$ be a connected graph. If $\mu$ and $\nu$ are unimodular measures sustained by $X$, then $\mu = \nu$.
\end{introthm}

The unique unimodular measure sustained by a connected judicial graph $X$ is denoted by $\Psi(X)$. This theorem relies on our useful, newly obtained criterion for unimodularity: a measure $\mu$ sustained by a connected graph $X$ is unimodular if and only if
\[
|G_a b|\mu[X,a] = |G_b a|\mu[X,b]
\]
for all adjacent vertices $a$ and $b$ of $X$ where $G = \Aut(X)$ and $G_{(\cdot)}$ is the stabilizer subgroup. Using this criterion, we determine the structure of a unimodular measure sustained by a connected graph.

\begin{introthm}{\ref{existence_of_sustained_unimodular_measures}}
Let $X$ be a connected graph; let $[X,1] \in \Rcc(X)$; let $p$ be a positive real number. Define a function $\mu : \Rcc(X) \to \RR$ as follows: $\mu[X,1] = p$, and for all $[X,x] \in \Rcc(X) \setminus \{[X,1]\}$,
\[
\mu[X,x] = \frac{|G_{x_0}x_1| |G_{x_1}x_2| \cdots |G_{x_{k-1}}x_k|}{|G_{x_1}x_0| |G_{x_2}x_1| \cdots |G_{x_k}x_{k-1}|} p
\]
where $(x_0,x_1,\ldots,x_k)$ is a path in $X$ with $x_0 = 1$ and $x_k = x$. Then the following statements hold.

\begin{enumerate}
\item The function $\mu$ is independent of the choice of path.

\item If $\mu(\Rcc(X)) = 1$, then $\mu$ is the unique unimodular measure sustained by $X$.
\end{enumerate}
\end{introthm}

In addition, we deduce that every Cayley graph is judicial, a result known to David Aldous and Russell Lyons \cite{aldouslyons07} in the setting of group unimodularity. Furthermore, we use Theorem \ref{uniqueness_of_sustained_unimodular_measures} to show that a unimodular measure sustained by a connected graph is an extreme point of the convex set $\U$ of unimodular measures.

We also deal with disconnected graphs, and determine when they are judicial. For example, if $b$ is a positive integer, the disjoint union $bX$ of $b$ copies of a connected judicial graph $X$ is judicial. In fact, $bX$ sustains the unique unimodular measure $\Psi(X)$. In general, however, a disconnected graph sustains multiple distinct unimodular measures, as the following new observation shows.

\begin{introthm}{\ref{judiciality_of_disconnected_graphs}}
Let $I$ be a subset of $\NN$. Suppose that $X = \sum_{k \in I} b_kX^k$ where $\{X^k ~:~ k \in I\}$ is a set of pairwise distinct connected judicial graphs and $\{b_k ~:~ k \in I\}$ is a set of positive integers. Let $\mu^k = \Psi(X^k)$ for all $k \in I$. If $\mu \in \U$ is sustained by $X$, then $\mu$ is the convex combination
\[
\mu = \sum_{k \in I} \left(\sum_{x \in \Rcc(k)} \mu[X^k,x]\right) \mu^k
\]
where $x \in \Rcc(k)$ is shorthand for $[X^k,x] \in \Rcc(X^k)$.
\end{introthm}

As a partial converse, a connected component of a judicial graph whose measure is nonzero is judicial too. Theorem \ref{judiciality_of_disconnected_graphs} is actually a generalization of the same result for finite disconnected graphs described in this author's Honours project \cite{artemenko11a}.

A sequence of graphs $(G_n)$ is negligible in the sequence $(X_n)$ of finite graphs if $G_n$ is a subgraph of $X_n$ for all positive integers $n$, and $|V(G_n)|/|V(X_n)| \to 0$ as $n \to \infty$. Intuitively, the subgraphs vanish in the limit. The following is a novel result for computing weak limits of sequences of complicated graphs. It says that the weak limit is preserved if we separate the graphs into simpler components.

\begin{introthm}{\ref{weak_limits_are_invariant_under_negligence}}
Suppose that $(X_n)$ and $(G_n)$ are sequences of finite graphs such that $(G_n)$ is negligible in $(X_n)$. Then
\[
\lim_{n \to \infty} \left|\int f ~d\Psi(X_n) - \int f ~d\Psi(X_n \setminus G_n)\right| = 0
\]
for all $f \in \C(\Gr)$. Furthermore, $(\Psi(X_n))$ converges weakly to $\mu$ if and only if $(\Psi(X_n \setminus G_n))$ does too.
\end{introthm}

Not all conjectures that sound reasonable in this area of research are true. Let $X$ be an infinite connected graph, and let $x \in V(X)$. The sequence $(\Psi(B_X(x,n)))$ of laws of closed balls with linearly increasing radii need not have a weak limit. Furthermore, it is possible to construct an infinite connected graph $X$ such that, even if $x$ and $y$ are adjacent vertices of $X$, the weak limits of $(\Psi(B_X(x,n)))$ and $(\Psi(B_X(y,n)))$ are distinct. Of course, we mention the famous open problem posed by G\'abor Elek \cite{elek07,elek10}, David Aldous, and Russell Lyons \cite{aldouslyons07}:

\begin{center}
``Is every unimodular measure the weak limit of a sequence of finite graphs?''
\end{center}

\noindent Lastly, we ask whether it is possible to determine if a countable connected graph is judicial without reference to a measure, and whether every extreme point of $\U$ is a unimodular measure sustained by a connected graph.

%% file: preliminaries.tex
%
%
\section{Preliminaries}

In this thesis, we adopt several conventions regarding concepts and notation. All of our graphs are assumed to be simple, undirected, countable (finite or infinite), and have at least one vertex, unless stated otherwise. Furthermore, we do not distinguish between a graph and its isomorphism class. That is, whenever $X$ and $Y$ are graphs, $X$ and $Y$ are isomorphic if and only if $X = Y$ where the equality is between the isomorphism classes of $X$ and $Y$. In addition, dotted lines in a figure of a graph mean the graph continues indefinitely. Lastly, the reader may assume that all of our measures are probability measures. Note that the results in this section are applied often enough that we use them without reference in the rest of the thesis.

Let $X$ be a graph, and let $A \subseteq V(X)$. The \emph{subgraph of $X$ induced by $A$}\index{subgraph of $X$ induced by $A$} is a graph whose vertex set is $A$, and whose edge set is $\{\{a,b\} \in E(X) ~:~ a,b \in A\}$. Denote by $X_x$\index{Xx@$X_x$} the \emph{connected component of $X$}\index{connected component} whose vertex set contains $x$. Define the graph metric $d_X$\index{dX@$d_X$} as follows: let $d_X(x,y)$ be the length of the shortest path from $x$ to $y$ in $X$ if $X$ is connected. For every nonnegative integer $r$ and $x \in V(X)$, the \emph{(closed) ball}\index{ball} $B_X(x,r)$\index{BXxr@$B_X(x,r)$} is the subgraph of $X$ induced by the set of vertices
\[
\{y \in V(X_x) ~:~ d_{X_x}(x,y) \leq r\},
\]
and the \emph{neighbourhood}\index{neighbourhood} $N_X(x) = B_X(x,1) \setminus \{x\}$\index{NXx@$N_X(x)$} is the set of vertices that are adjacent to $x$. In each case, the subscript $X$ may be dropped if there is no confusion. A \emph{walk}\index{walk} in a graph $X$ is a sequence
\[
(x_0,x_1,\ldots,x_k)
\]
of vertices such that $x_i \in N(x_{i+1})$ for all $i \in \{0,1,\ldots,k - 1\}$; a \emph{path}\index{path} is a walk whose vertices are pairwise distinct. If $X$ is a graph, then $\Aut(X)$\index{AutX@$\Aut(X)$} is its \emph{automorphism group}\index{automorphism group}. For all $x \in V(X)$, $\Aut(X)x$\index{AutXx@$\Aut(X)x$} is the \emph{orbit of $x$}\index{orbit of $x$} under the action of $\Aut(X)$. If $X$ and $Y$ are graphs, and $k$ is a positive integer, then $X + Y$\index{XY@$X + Y$} is the disjoint union of $X$ and $Y$, and $kX$\index{kX@$kX$} is the disjoint union of $k$ copies of $X$. Furthermore, let $G$ be a subgraph of $X$. Then $X \setminus G$\index{XG@$X \setminus G$} is the subgraph of $X$ induced by $V(X) \setminus V(G)$.

%
%
\subsection{Rooted and birooted graphs}

\begin{defn}
A \emph{rooted graph}\index{rooted graph} is a pair $(X,x)$ where $X$ is a graph and $x \in V(X)$; a \emph{birooted graph}\index{birooted graph} is an ordered triple $(X,x,y)$ where $X$ is a graph, $x \in V(X)$, and $y \in N_X(x)$.
\end{defn}

Fix a positive integer $\Delta$\index{Delta@$\Delta$}, which remains the same throughout this entire thesis. Let $\Gr$\index{Gr@$\Gr$} be the set of all isomorphism classes $[X,x]$ of countable, connected rooted graphs $(X,x)$ such that $\deg_X(y) \leq \Delta$ for all $y \in V(X)$.

Define the metric $\rho : \Gr \times \Gr \to \RR$\index{rho@$\rho$} as follows:
\[
\rho([X,x],[Y,y]) =
\begin{cases}
0 & \text{ if $[X,x] = [Y,y]$,}\\
2^{-r} & \text{ otherwise}
\end{cases}
\]
where $r = \sup\{s \in \NN ~:~ [B_X(x,s),x] = [B_Y(y,s),y]\}$. It is known that $(\Gr,\rho)$ is a compact ultrametric space \cite{artemenko11a}.

Similarly, $\dGr$\index{Gr@$\dGr$} is the set of all isomorphism classes $[X,x,y]$ of countable, connected birooted graphs $(X,x,y)$ such that $\deg_X(z) \leq \Delta$ for all $z \in V(X)$. There is an analogous ultrametric $\vec\rho$\index{rho@$\vec\rho$} on $\dGr$, but it is not used in this thesis.

Furthermore, we equip $\Gr$ and $\dGr$ with the Borel $\sigma$-algebras of $\rho$ and $\vec\rho$, respectively. The reader is welcome to peruse this author's previous works \cite{artemenko11a,artemenko11b} to learn more about the compact ultrametric spaces $\Gr$ and $\dGr$.

\begin{defn}
If $X$ is a graph and $x \in V(X)$, then $[X_x,x]$ is a \emph{rooted connected component of $X$}\index{rooted connected component}. Denote by $\Rcc(X)$\index{RccX@$\Rcc(X)$} the set of rooted connected components of $X$. Similarly, $[X_x,x,y]$ is a \emph{birooted connected component of $X$}\index{birooted connected component} for all adjacent vertices $x$ and $y$ of $X$, and $\BRcc(X)$\index{BRccX@$\BRcc(X)$} is the set of birooted connected components of $X$.
\end{defn}

\begin{prop}
Let $X$ be a graph, and let $I$ be a subset of $\NN$. If $X = \sum_{k \in I} X^k$ where $X^k$ is a connected graph for all $k \in I$, then $\Rcc(X) = \bigcup_{k \in I} \Rcc(X^k)$.
\end{prop}

\begin{proof}
Suppose that $[X_x,x] \in \Rcc(X)$. Then $X_x = X^k$ for some $k \in I$. That is, $[X_x,x] = [X^k,x] \in \Rcc(X^k)$. On the other hand, assume that $[X^k,x] \in \Rcc(X^k)$ for some $k \in I$. Since $x \in V(X^k)$, $X^k = X_x$, and so $[X^k,x] = [X_x,x] \in \Rcc(X)$.
\end{proof}

\begin{prop}
Let $X$ be a connected graph, and let $[X,a,b],[X,c,d] \in \dGr$. If $[X,a,b] = [X,c,d]$, then $[X,a] = [X,c]$ and $[X,b] = [X,d]$.
\end{prop}

\begin{proof}
Suppose that $[X,a,b] = [X,c,d]$. There is an automorphism $\varphi : X \to X$ such that $\varphi(a) = c$ and $\varphi(b) = d$. Hence $[X,a] = [X,c]$ and $[X,b] = [X,d]$.
\end{proof}

The following proposition is here to justify our ability to define graphs using their rooted connected components, to claim that two graphs are equal precisely when they share at least one rooted connected component, and other technicalities.

\begin{prop}\label{equiv_of_rcc}
Let $X$ and $Y$ be connected graphs. The following statements are equivalent: (i) $X = Y$, (ii) $\Rcc(X) = \Rcc(Y)$, and (iii) $\Rcc(X) \cap \Rcc(Y) \neq \emptyset$.
\end{prop}

\begin{proof}
To prove that $(i)$ implies $(ii)$, assume that $X = Y$. Let $[X,x] \in \Rcc(X)$. Since $X = Y$, there is an isomorphism $f : X \to Y$, and so $Y = f(X)$. Then
\[
[X,x] = [Y,f(x)] \in \Rcc(Y).
\]
That is, $\Rcc(X) \subseteq \Rcc(Y)$. Similarly, $\Rcc(Y) \subseteq \Rcc(X)$, and $(ii)$ follows. The implication $(ii)$ implies $(iii)$ is trivial. Lastly, $(iii)$ implies $(i)$ may be shown as follows. Consider a rooted connected component in the intersection of $\Rcc(X)$ and $\Rcc(Y)$: $[X,x] = [Y,y]$ for some $x \in V(X)$ and $y \in V(Y)$. Then there is an isomorphism $f : X \to Y$ such that $f(x) = y$. In particular, $X = Y$.
\end{proof}

Proposition \ref{equiv_of_rcc} only considers connected graphs. However, there are equally useful facts for graphs that may not be connected. In particular, the following results tell us how orbits are related to rooted connected components.

\begin{lem}\label{component_to_component}
Let $X$ and $Y$ be graphs. If $f : X \to Y$ is an isomorphism, then $f(X_x) = Y_{f(x)}$, and so $[X_x,x] = [Y_{f(x)},f(x)]$ for all $x \in V(X)$.
\end{lem}

\begin{proof}
Let $x \in V(X)$. Since $f$ is an isomorphism, the image $f(X_x)$ of the connected component $X_x$ is itself connected. Let us prove that $f(X_x) = Y_{f(x)}$.

Suppose that $y \in f(X_x)$. Then $y = f(z)$ for some $z \in X_x$. Since $X_x$ is connected, there is a path $P$ between $x$ and $z$ in $X_x$. The image $f(P)$ is a path between $f(x)$ and $f(z) = y$, and so $y \in Y_{f(x)}$. Thus $f(X_x) \subseteq Y_{f(x)}$. On the other hand, assume that $y \in Y_{f(x)}$. Consider the inverse $g : Y \to X$ of the isomorphism $f$. As before, the image of a path between $y$ and $f(x)$ is a path between $g(y)$ and $g(f(x)) = x$. Thus $g(y) \in X_x$, and so $y = f(g(y)) \in f(X_x)$. That is, $Y_{f(x)} \subseteq f(X_x)$.
\end{proof}

\begin{prop}\label{same_orbit_same_rcc}
Let $X$ be a graph. Suppose that $x$ and $y$ are vertices of $X$. Then $y \in \Aut(X)x$ if and only if $[X_x,x] = [X_y,y]$.
\end{prop}

\begin{proof}
If $y \in \Aut(X)x$, there is an automorphism $\varphi$ of $X$ such that $\varphi(x) = y$. By Lemma \ref{component_to_component}, $[X_x,x] = [\varphi(X_x),\varphi(x)] = [X_y,y]$. Conversely, $[X_x,x] = [X_y,y]$ implies there is an isomorphism $\varphi : X_x \to X_y$ such that $\varphi(x) = y$. If $x$ and $y$ lie in the same component, then $X_x = X_y$, and $\hat\varphi : X \to X$ defined by
\[
\hat\varphi(z) =
\begin{cases}
\varphi(z)      & \text{ if $z \in V(X_x)$,}\\
z               & \text{ otherwise}
\end{cases}
\]
for all $z \in V(X)$ is an extension of $\varphi$. On the other hand, assume that $x$ and $y$ belong to distinct components. It is possible to define an extension $\hat\varphi$ on $X$ as follows:
\[
\hat\varphi(z) =
\begin{cases}
\varphi(z)      & \text{ if $z \in V(X_x)$,}\\
\varphi^{-1}(z) & \text{ if $z \in V(X_y)$,}\\
z               & \text{ otherwise}
\end{cases}
\]
for all $z \in V(X)$. In either case, it is easy to see that $\hat\varphi$ is an automorphism of $X$, and $\hat\varphi(x) = y$. That is, $y \in \Aut(X)x$.
\end{proof}

%
%
\subsection{Sustained and unimodular measures}

\begin{defn}\cite{artemenko11a}
A measure $\mu$ on $\Gr$ is \emph{sustained}\index{sustained} by a graph $X$ if $\Rcc(X)$ has full measure, that is, $\mu(\Gr \setminus \Rcc(X)) = 0$. Alternatively, $X$ \emph{sustains}\index{sustains} $\mu$. It is \emph{strictly sustained}\index{strictly sustained} by $X$ if every rooted connected component has a positive measure.
\end{defn}

\begin{prop}
If $\mu$ and $\nu$ are measures sustained by a graph $X$, then $\mu = \nu$ if and only if $\mu[X_x,x] = \nu[X_x,x]$ for all $[X_x,x] \in \Rcc(X)$.
\end{prop}

\begin{proof}
Since $\mu$ and $\nu$ are sustained by $X$, the $\mu$-measure and $\nu$-measure of $\Gr \setminus \Rcc(X)$ is zero. Thus $\mu[Y,y] = 0 = \nu[Y,y]$ for all $[Y,y] \notin \Rcc(X)$. Since $\mu$ and $\nu$ always agree on the complement of $\Rcc(X)$, the result follows.
\end{proof}

The main concept of this thesis is \emph{unimodularity}\index{unimodularity}. Since its introduction by Benjamini and Schramm \cite{benjaminischramm01}, and Aldous and Steele \cite{aldoussteele03}, it has been reformulated and renamed in several papers. Our attempt is also known as \emph{involution invariance}\index{involution invariance} according to Aldous and Lyons \cite[p.~10]{aldouslyons07}. However, the paper by Aldous and Steele \cite[p.~40]{aldoussteele03} provides an alternative, yet equivalent, definition with this terminology. Another variation is the \emph{intrinsic mass transport principle}\index{intrinsic mass transport principle}, which is shown to coincide with our definition \cite[p.~11]{aldouslyons07}.

\begin{defn}
A measure $\mu$ on $\Gr$ is \emph{unimodular}\index{unimodular} if
\[
\int \sum_{y \in N(x)} f[X_x,x,y] ~d\mu[X_x,x] = \int \sum_{y \in N(x)} f[X_x,y,x] ~d\mu[X_x,x]
\]
for all nonnegative measurable functions $f$ on $\dGr$. The set of unimodular measures on $\Gr$ is denoted by $\U$\index{U@$\U$}.
\end{defn}

To better understand the concept of unimodularity, think of the integral of a function $g$ with respect to the measure $\mu$ as the average value of $g$. The function value $f[X_x,x,y]$ is the amount of ``mass'' being sent from $x$ to $y$. Then $\mu$ is unimodular if the average amount of mass sent \emph{from} a vertex to its neighbours equals the average amount of mass sent \emph{to} a vertex from its neighbours.

Since every graph in this thesis is countable, and the measures we deal with are sustained by such graphs, it is wise to just use the following reformulation instead.

\begin{prop}
Let $X$ be a graph. A measure $\mu$ sustained by $X$ is unimodular if and only if
\[
\sum_{[X_x,x] \in \Rcc(X)} \sum_{y \in N(x)} f[X_x,x,y] \cdot \mu[X_x,x] = \sum_{[X_x,x] \in \Rcc(X)} \sum_{y \in N(x)} f[X_x,y,x] \cdot \mu[X_x,x] \tag{$\star$}
\]
for all nonnegative functions $f$ on $\dGr$ with $f(\dGr \setminus \BRcc(X)) = \{0\}$.
\end{prop}

\begin{proof}
Let $\mu$ be a measure sustained by a graph $X$. Clearly, $\mu$ is unimodular if and only if ($\star$) holds for all nonnegative measurable functions $f$ on $\dGr$ because $\mu$ is sustained by $X$. Suppose that $\mu$ is unimodular. Let $f$ be a nonnegative function on $\dGr$ with $f(\dGr \setminus \BRcc(X)) = \{0\}$. Since $X$ is countable, so is $\BRcc(X)$. It follows that $f$ is measurable, and the result holds. Conversely, assume that $f$ is a nonnegative measurable function on $\dGr$. Let $g$ be the product of $f$ and the characteristic function of $\BRcc(X)$. It is obvious that $g$ is a nonnegative function on $\dGr$ and $g(\dGr \setminus \BRcc(X)) = \{0\}$, which means $g$ satisfies ($\star$). Then
\begin{align*}
\sum_{[X_x,x] \in \Rcc(X)} \sum_{y \in N(x)} f[X_x,x,y] \cdot \mu[X_x,x] &= \sum_{[X_x,x] \in \Rcc(X)} \sum_{y \in N(x)} g[X_x,x,y] \cdot \mu[X_x,x]\\
        &= \sum_{[X_x,x] \in \Rcc(X)} \sum_{y \in N(x)} g[X_x,y,x] \cdot \mu[X_x,x]\\
        &= \sum_{[X_x,x] \in \Rcc(X)} \sum_{y \in N(x)} f[X_x,y,x] \cdot \mu[X_x,x]
\end{align*}
because $f = g$ on $\BRcc(X)$, and so $\mu$ is unimodular.
\end{proof}

\begin{rem}
In arguments involving unimodularity and a measure $\mu$ sustained by a graph $X$, it is usually assumed, without loss of generality, that $f(\dGr \setminus \BRcc(X)) = \{0\}$ whenever $f$ is a nonnegative function on $\dGr$.
\end{rem}

To combine the concepts of ``sustained'' and ``unimodular,'' we introduce new terminology to distinguish graphs that sustain unimodular measures.

\begin{defn}
A graph $X$ is \emph{judicial}\index{judicial} if there is a unimodular measure $\mu$ sustained by $X$. If a graph is not judicial, it is said to be \emph{lawless}\index{lawless}.
\end{defn}

\begin{defn}\cite{schramm08}
Let $X$ be a finite graph. Define the measure $\Psi(X)$\index{PsiX@$\Psi(X)$!law of a finite graph} on $\Gr$ by
\[
\Psi(X)[X_x,x] = \frac{|\Aut(X)x|}{|V(X)|}
\]
for all $[X_x,x] \in \Rcc(X)$, and $\Psi(X) = 0$ otherwise. The measure $\Psi(X)$ is clearly sustained by $X$, and it is known as the \emph{law of $X$}\index{law of $X$}.
\end{defn}

\begin{prop}\label{integral_wrt_law}
If $X$ is a finite graph, then
\[
\int g ~d\Psi(X) = \sum_{x \in V(X)} \frac{g[X_x,x]}{|V(X)|}
\]
for all measurable functions $g : \Gr \to \RR$.
\end{prop}

\begin{proof}
Let $g : \Gr \to \RR$ be a measurable function. By Proposition \ref{same_orbit_same_rcc},
\[
\{\Aut(X)x ~:~ [X_x,x] \in \Rcc(X)\}
\]
is a partition of $V(X)$. Then
\[
\sum_{x \in V(X)} \frac{g[X_x,x]}{|V(X)|} = \sum_{[X_x,x] \in \Rcc(X)} \sum_{y \in \Aut(X)x} \frac{g[X_y,y]}{|V(X)|} = \sum_{[X_x,x] \in \Rcc(X)} g[X_x,x] \cdot \frac{|\Aut(X)x|}{|V(X)|}
\]
where the second equality holds by Proposition \ref{same_orbit_same_rcc}, and the result follows.
\end{proof}

\begin{prop}\cite{schramm08}
The law $\Psi(X)$ of a finite graph $X$ is unimodular. In particular, every finite graph is judicial.
\end{prop}

\begin{proof}
Let $X$ be a finite graph, and let $f$ be a nonnegative function on $\dGr$. Clearly, the set
\[
S = \{(a,b) \in V(X) \times V(X) ~:~ b \in N(a)\}
\]
is symmetric. That is, $(a,b) \in S$ if and only if $(b,a) \in S$ for all vertices $a$ and $b$ of $X$. In addition, $a \in V(X)$ and $b \in N(a)$ if and only if $(a,b) \in S$. Furthermore,
\[
\sum_{(x,y) \in S} f[X_x,x,y] = \sum_{(y,x) \in S} f[X_x,x,y] = \sum_{(x,y) \in S} f[X_y,y,x] = \sum_{(x,y) \in S} f[X_x,y,x] \tag{$\star$}
\]
where $X_x = X_y$ because $x$ and $y$ are neighbours, and the second equality is obtained by relabelling $x$ and $y$. Using Proposition \ref{integral_wrt_law},
\[
\sum_{[X_x,x] \in \Rcc(X)} \sum_{y \in N(x)} f[X_x,x,y] \cdot \Psi(X)[X_x,x] = \sum_{x \in V(X)} \sum_{y \in N(x)} \frac{f[X_x,x,y]}{|V(X)|}
\]
and
\[
\sum_{[X_x,x] \in \Rcc(X)} \sum_{y \in N(x)} f[X_x,y,x] \cdot \Psi(X)[X_x,x] = \sum_{x \in V(X)} \sum_{y \in N(x)} \frac{f[X_x,y,x]}{|V(X)|}.
\]
According to ($\star$), the right-hand sides of these two equations are equal, which means $\Psi(X)$ is unimodular. Hence $X$ is judicial.
\end{proof}

%
%
\subsection{Weak convergence}

Let $\C(\Gr)$ be the set of bounded continuous real-valued functions on $\Gr$. As defined in measure theory, weak convergence plays a prominent role in this thesis. Its definition is restricted to measures on $\Gr$ for simplicity.

\begin{defn}
A sequence $(\mu_n)$ of measures on $\Gr$ \emph{converges weakly}\index{converges weakly} to some measure $\mu$ on $\Gr$ if
\[
\int f ~d\mu_n \to \int f ~d\mu
\]
for all $f \in \C(\Gr)$. The measure $\mu$ is known as the \emph{weak limit}\index{weak limit} of the given sequence.

As a special case, assume that $(\Psi(X_n))$ is a sequence of laws of finite graphs. The sequence $(X_n)$ of finite graphs \emph{converges weakly} if $(\Psi(X_n))$ does too.
\end{defn}

The proof of the following important fact, which is due to Benjamini and Schramm \cite{benjaminischramm01}, is omitted in this thesis. However, the reader is encouraged to see this author's paper on graphings and unimodularity for a full treatment \cite{artemenko11b}. As stated by Schramm \cite{schramm08}, Aldous, and Lyons \cite{aldouslyons07}, the converse is still an open question.

\begin{prop}
If $\mu$ is the weak limit of a sequence $(\Psi(X_n))$ where $X_n$ is a finite graph for each positive integer $n$, then $\mu$ is unimodular.\qed
\end{prop}

%
%
\subsection{Examples of graphs}

\begin{defn}
A graph $X$ is \emph{vertex-transitive}\index{vertex-transitive} if $V(X) = \Aut(X)x$ for some $x \in V(X)$; a graph $X$ is \emph{rigid}\index{rigid} if its automorphism group $\Aut(X)$ is trivial.
\end{defn}

\begin{prop}
Let $X$ be a connected graph. If $[X,a,b] = [X,b,a]$ for all $a \in V(X)$ and $b \in N(a)$, then $X$ is vertex-transitive.
\end{prop}

\begin{proof}
Let $x,y \in V(X)$. Since $X$ is connected, there is a path $(x_0,x_1,\ldots,x_k)$ in $X$ with $x_0 = x$ and $x_k = y$. By assumption,
\[
[X,x_i,x_{i+1}] = [X,x_{i+1},x_i]
\]
for all $i \in \{0,1,\ldots,k - 1\}$, which means there are automorphisms $\varphi_1,\varphi_2,\ldots,\varphi_k$ of $X$ such that $\varphi_j(x_{j-1}) = x_j$ for all $j \in \{1,2,\ldots,k\}$. Hence $\varphi = \varphi_k \circ \varphi_{k-1} \circ \cdots \circ \varphi_1$ is an automorphism of $X$ such that $\varphi(x) = y$.
\end{proof}

There are several important examples of finite and infinite graphs mentioned in this thesis. Each of these graphs is described below, and the proposition that follows gives a taste of their usefulness.

\begin{figure}[ht]
\begin{center}
\includegraphics[scale=1.2]{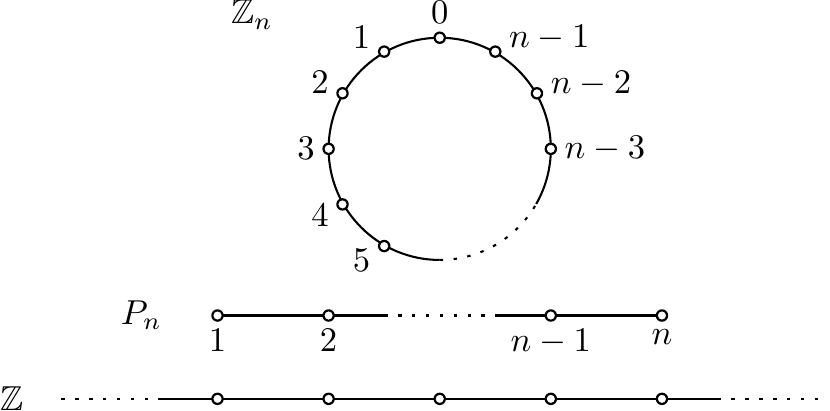}
\caption[The graphs $\ZZ_n$, $P_n$, and $\ZZ$]{The graphs $\ZZ_n$, $P_n$, and $\ZZ$.}
\label{cycles_and_paths}
\end{center}
\end{figure}

\begin{eg}
The graph $\ZZ_n$\index{Zn@$\ZZ_n$} is the \emph{cycle on $n$ vertices}, and $P_n$\index{Pn@$P_n$} is the \emph{path on $n$ vertices}. Additionally, $\ZZ$\index{Z@$\ZZ$} is known as the \emph{bi-infinite path}\index{bi-infinite path}; its vertices are the integers, and its edge set consists of the elements $\{i,i+1\}$ for all $i \in \ZZ$. These three types of graphs are presented in Figure \ref{cycles_and_paths}.
\end{eg}

\begin{eg}
The graph $T$\index{T@$T$} in Figure \ref{3_regular_tree} is the \emph{$3$-regular tree}\index{3regulartree@$3$-regular tree}, which is an infinite vertex-transitive tree each of whose vertices has degree $3$. The graph $\Lambda$\index{Lambda@$\Lambda$} in Figure \ref{infinite_perfect_binary_tree} is the \emph{infinite perfect binary tree}\index{infinite perfect binary tree}. The graph\footnote{The reader should be aware that $\bar\Lambda$ is \emph{not} a tree because it is not acyclic. The terminology is used to establish a link with the infinite perfect binary tree.} $\bar\Lambda$\index{Lambda@$\bar\Lambda$} in Figure \ref{barred_binary_tree} is the \emph{barred (infinite perfect) binary tree}\index{barred binary tree}. The \emph{first ancestor}\index{first ancestor} of $\Lambda$ (and $\bar\Lambda$) is the unique vertex of degree $2$.
\end{eg}

\begin{figure}[ht]
\begin{center}
\includegraphics[scale=1.2]{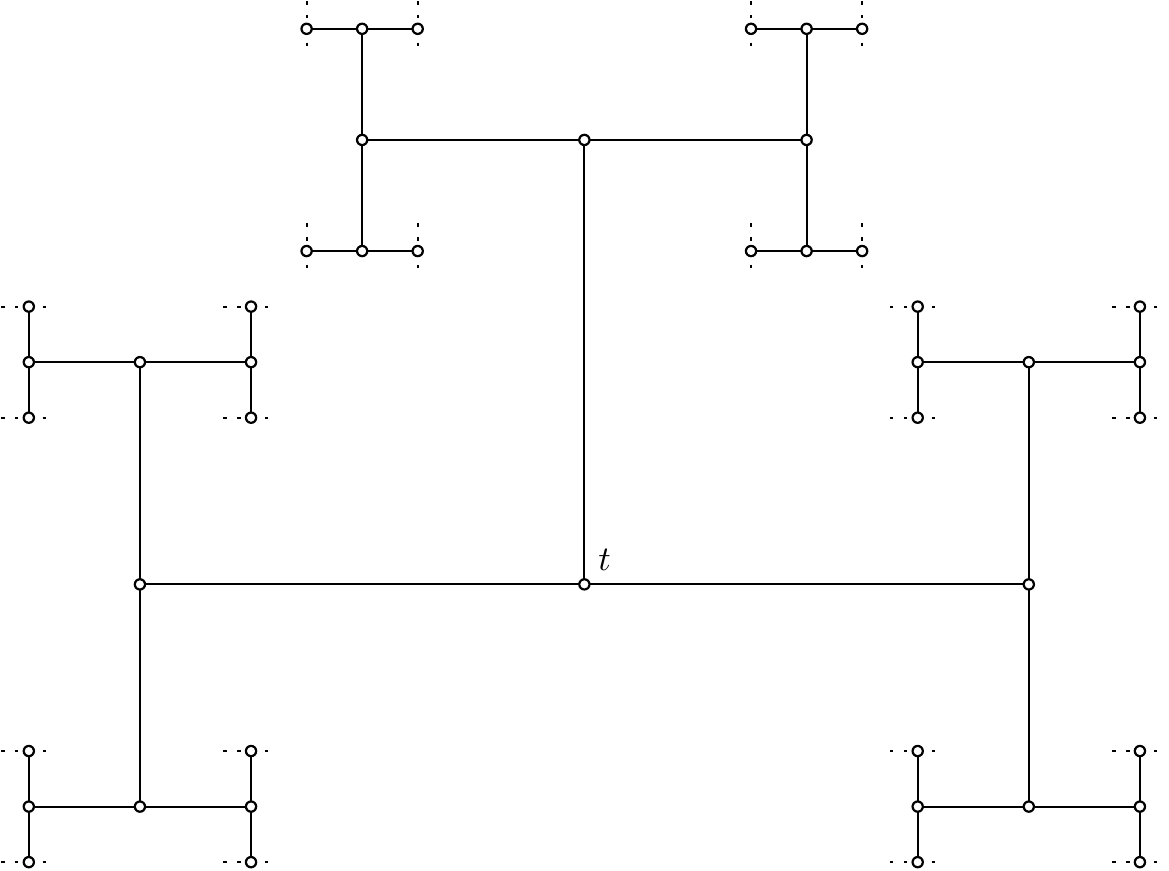}
\caption[The $3$-regular tree]{The $3$-regular tree, and a specified vertex $t$.}
\label{3_regular_tree}
\end{center}
\end{figure}

\begin{figure}[ht]
\begin{center}
\includegraphics[scale=1.2]{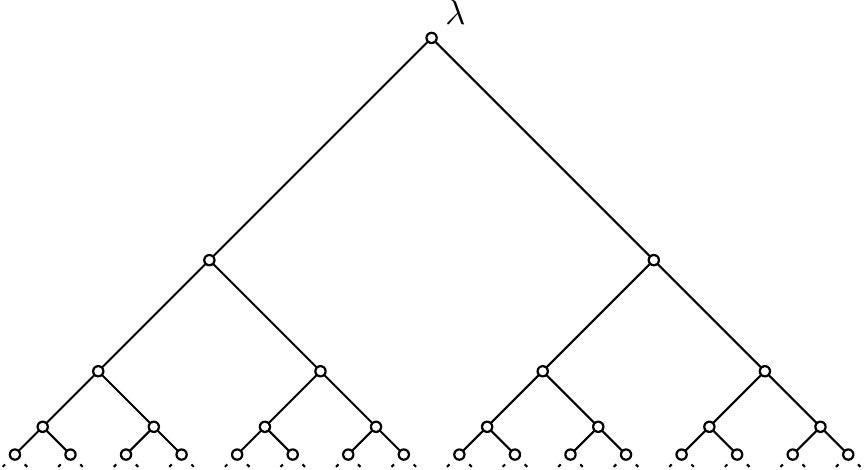}
\caption[The infinite perfect binary tree]{The infinite perfect binary tree; the vertex $\lambda$ is its first ancestor.}
\label{infinite_perfect_binary_tree}
\end{center}
\end{figure}

\begin{figure}[ht]
\begin{center}
\includegraphics[scale=1.2]{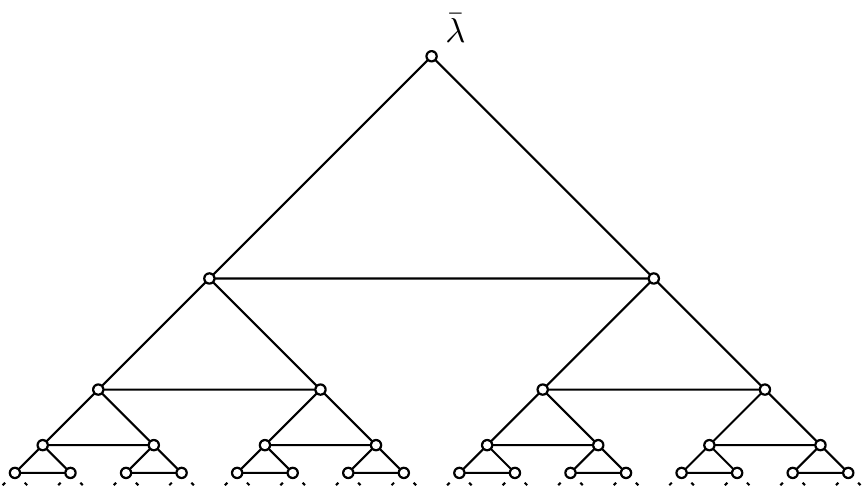}
\caption[The barred binary tree]{The barred binary tree; the vertex $\bar\lambda$ is its first ancestor.}
\label{barred_binary_tree}
\end{center}
\end{figure}

\begin{figure}[ht]
\begin{center}
\includegraphics[scale=1.2]{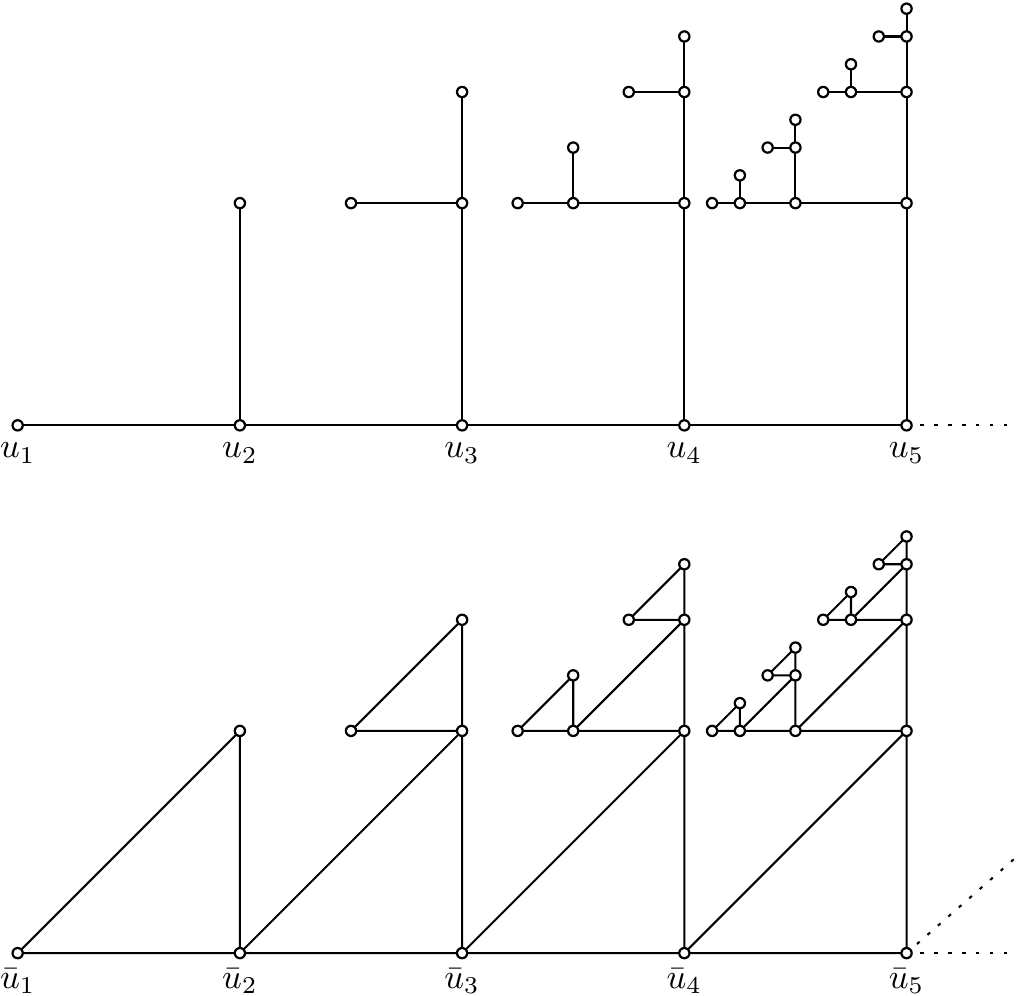}
\caption[The graphs $S$ and $\bar{S}$]{The graphs $S$ and $\bar{S}$; (metric) limits of the balls in $T$ and $\bar\Lambda$, respectively.}
\label{s_and_bar_s}
\end{center}
\end{figure}

Given a positive integer $n$, let $T_n = B_T(t,n)$ and $\bar\Lambda_n = B_{\bar\Lambda}(\bar\lambda,n)$ be the balls in $T$ and $\bar\Lambda$, respectively. There is a unique path $(u_1,u_2,\ldots,u_n,t)$ in $T_n$ where $u_1$ is a leaf. Similarly, there is a unique path $(\bar{u}_1,\bar{u}_2,\ldots,\bar{u}_n,\bar\lambda)$ in $\bar\Lambda_n$ where $\bar{u}_1$ is a vertex of degree $2$, and the vertices of the path belong to pairwise distinct orbits. With this, the graphs $S$\index{S@$S$} and $\bar{S}$\index{S@$\bar{S}$} depicted in Figure \ref{s_and_bar_s} are defined as follows: $[S,u_i] = \lim_{n \to \infty} [T_n,u_i]$ and $[\bar{S},\bar{u}_i] = \lim_{n \to \infty} [\bar\Lambda_n,\bar{u}_i]$ for all positive integers $i$.

\begin{lem}\label{term_conv_implies_sum_conv}
Let $a_n^i$ and $a^i$ be real numbers for all positive integers $n$ and $i$. Suppose that $\lim_{n \to \infty} a_n^i = a^i$, and assume that there is a real number $L$ such that $|a_n^i| \leq L$ and $|a^i| \leq L$ for all positive integers $n$ and $i$. Then
\[
\lim_{n \to \infty} \sum_{i=1}^n \frac{a_n^i}{2^i} = \sum_{i=1}^\infty \frac{a^i}{2^i}.
\]
\end{lem}

\begin{proof}
Let $\e$ be a positive real number. There is a positive integer $M$ such that $4L/2^M < \e$. For each $i \in \{1,2,\ldots,M\}$, there is a positive integer $N_i$ such that
\[
|a_n^i - a^i| < \frac{\e}{2}
\]
for all positive integers $n$ with $n \geq \max\{N_i,M + 1\}$. Choose
\[
N = \max\{N_1,N_2,\ldots,N_M,M + 1\},
\]
and let $n$ be a positive integer with $n \geq N$. Then
\begin{align*}
\left|\sum_{i=1}^n \frac{a_n^i}{2^i} - \sum_{i=1}^\infty \frac{a^i}{2^i}\right| &\leq \sum_{i=1}^M \frac{|a_n^i - a^i|}{2^i} + \sum_{i=M+1}^\infty \frac{|a_n^i|}{2^i} + \sum_{i=M+1}^\infty \frac{|a^i|}{2^i}\\
                       &\leq \sum_{i=1}^M \frac{|a_n^i - a^i|}{2^i} + \sum_{i=M+1}^\infty \frac{2L}{2^i}\\
                       &\leq \sum_{i=1}^M \frac{|a_n^i - a^i|}{2^i} + \frac{2L}{2^M} < \left(\frac{\e}{2}\right) \sum_{i=1}^M \frac{1}{2^i} + \frac{2L}{2^M}\\
                       &< \frac{\e}{2} + \frac{\e}{2} = \e,
\end{align*}
and the result follows.
\end{proof}

\begin{prop}\label{weak_limit_of_trees_barred_trees}
Let $T_n = B_T(t,n)$ and $\bar\Lambda_n = B_{\bar\Lambda}(\bar\lambda,n)$ for all positive integers $n$. The weak limits of $(\Psi(T_n))$ and $(\Psi(\bar\Lambda_n))$ are $\mu_S$ and $\mu_{\bar{S}}$, respectively, defined by
\[
\mu_S[S,u_i] = \frac{1}{2^i} = \mu_{\bar{S}}[\bar{S},\bar{u}_i]
\]
for all positive integers $i$.
\end{prop}

\begin{proof}
Let $f \in \C(\Gr)$. By the continuity of $f$, and the definitions of $S$ and $\bar{S}$,
\[
f[S,u_i] = \lim_{n \to \infty} f[T_n,u_i]
\]
and
\[
f[\bar{S},\bar{u}_i] = \lim_{n \to \infty} f[\bar\Lambda_n,\bar{u}_i]
\]
for all positive integers $i$. Using induction, it is clear that $|V(T_n)| = 3 \cdot 2^n - 2$, $|V(\bar\Lambda)| = 2^{n+1} - 1$, $|\Aut(T_n)u_i| = 3 \cdot 2^{n-i}$, and $|\Aut(\bar\Lambda_n)\bar{u}_i| = 2^{n+1-i}$ for all positive integers $n$ and $i \in \{1,2,\ldots,n\}$. Then
\begin{align*}
\lim_{n \to \infty} \int f ~d\Psi(T_n) &= \lim_{n \to \infty} \left(\sum_{i=1}^n \frac{f[T_n,u_i] \cdot |\Aut(T_n)u_i|}{|V(T_n)|} + \frac{f[T_n,t] \cdot |\Aut(T_n)t|}{|V(T_n)|}\right)\\
                                       &= \lim_{n \to \infty} \left(\sum_{i=1}^n \frac{f[T_n,u_i] \cdot 3 \cdot 2^{n-i}}{3 \cdot 2^n - 2} + \frac{f[T_n,t]}{3 \cdot 2^n - 2}\right)\\
                                       &= \lim_{n \to \infty} \left(\sum_{i=1}^n \frac{f[T_n,u_i]}{2^i}\right) = \sum_{i=1}^\infty \frac{f[S,u_i]}{2^i} = \int f ~d\mu_S
\end{align*}
and
\begin{align*}
\lim_{n \to \infty} \int f ~d\Psi(\bar\Lambda_n) &= \lim_{n \to \infty} \left(\sum_{i=1}^n \frac{f[\bar\Lambda_n,\bar{u}_i] \cdot |\Aut(\bar\Lambda_n)\bar{u}_i|}{|V(\bar\Lambda_n)|} + \frac{f[\bar\Lambda_n,\bar\lambda] \cdot |\Aut(\bar\Lambda_n)\bar\lambda|}{|V(\bar\Lambda_n)|}\right)\\
                                                 &= \lim_{n \to \infty} \left(\sum_{i=1}^n \frac{f[\bar\Lambda_n,\bar{u}_i] \cdot 2^{n+1-i}}{2^{n+1} - 1} + \frac{f[\bar\Lambda_n,\bar\lambda]}{2^{n+1} - 1}\right)\\
                                                 &= \lim_{n \to \infty} \left(\sum_{i=1}^n \frac{f[\bar\Lambda_n,\bar{u}_i]}{2^i}\right) = \sum_{i=1}^\infty \frac{f[\bar{S},\bar{u}_i]}{2^i} = \int f ~d\mu_{\bar{S}}
\end{align*}
where, in either case, the fourth equality holds by Lemma \ref{term_conv_implies_sum_conv}, which applies because $f$ is bounded. Hence the weak limit of $(\Psi(T_n))$ is $\mu_S$, and the weak limit of $(\Psi(\bar\Lambda_n))$ is $\mu_{\bar{S}}$.
\end{proof}

%% file: uniqueness_and_existence_of_sustained_unimodular_measures.tex
%
%
\section{Uniqueness and existence of sustained unimodular measures}

The purpose of this section is to demonstrate that unimodular measures sustained by connected graphs are unique. More precisely, a connected graph cannot sustain more than one unimodular measure. Following the proof of this result, we discuss its relation to Cayley graphs, and its application to the study of the convexity of $\U$. Additionally, we determine the structure of a unimodular measure sustained by a connected graph.

%
%
\subsection{Infinite connected rigid graphs are lawless}

Recall that a finite connected graph always sustains a unimodular measure, known as its law. However, this is not necessarily true for infinite connected graphs, as the following proposition proves.

\begin{prop}\label{infinite_connected_rigid_graphs_not_judicial}
If $X$ is an infinite connected rigid graph, then there is no unimodular measure sustained by $X$.
\end{prop}

\begin{proof}
Suppose that $\mu$ is a unimodular measure sustained by $X$. Since $X$ is rigid, there is a unique birooted graph $[X,x,y]$ for each pair of neighbours $(x,y)$ in $X$. That is, $[X,x_1,y_1] = [X,x_2,y_2]$ if and only if $(x_1,y_1) = (x_2,y_2)$ for all adjacent vertices $x_1$ and $y_1$, and $x_2$ and $y_2$ of $X$. Let $a \in V(X)$ and $b \in N(a)$. Denote by $f$ the characteristic function of the singleton $\{[X,a,b]\}$. By the unimodularity of $\mu$,
\[
\sum_{x \in V(X)} \sum_{y \in N(x)} f[X,x,y] \cdot \mu[X,x] = \sum_{x \in V(X)} \sum_{y \in N(x)} f[X,y,x] \cdot \mu[X,x],
\]
and so $\mu[X,a] = \mu[X,b]$ because $X$ is rigid. Hence there is a real number $k$ such that $\mu[X,x] = k$ for all $x \in V(X)$. Then
\[
1 = \mu(\Gr) = \sum_{x \in V(X)} \mu[X,x] = \sum_{x \in V(X)} k,
\]
which is a contradiction.
\end{proof}

%
%
\subsection{Example: the graph $S$ where $\Rcc(S)$ is infinite}

\begin{figure}[ht]
\begin{center}
\includegraphics[scale=1.2]{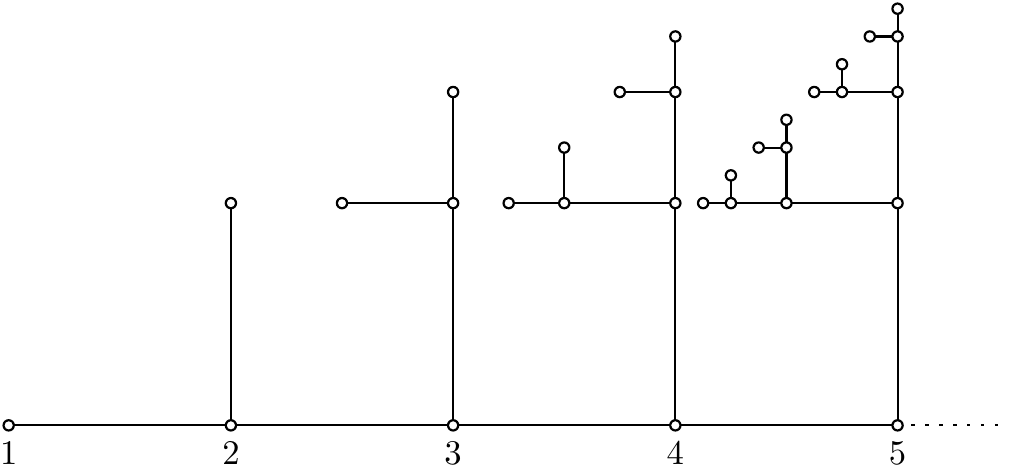}
\caption[The graph $S$]{The graph $S$.}
\label{the_graph_s}
\end{center}
\end{figure}

Consider the graph $S$ in Figure \ref{the_graph_s}. For convenience, let $\Rcc(S) = \{[S,i] ~:~ i \geq 1\}$. Suppose that $\mu$ is a measure on $\Gr$ such that $\sum_{i=1}^\infty \mu[S,i] = 1$. It is valid to ask how $\mu$ may be defined. An example of a unimodular measure $\mu_S$ sustained by $S$ is described in Proposition \ref{weak_limit_of_trees_barred_trees} where $\mu_S$ is defined by $\mu_S[S,i] = 2^{-i}$ for all positive integers $i$. In fact, this is the only unimodular measure sustained by $S$.

To understand why, recall that such a measure $\mu$ is unimodular if and only if
\[
\sum_{i=1}^\infty \sum_{j \in N(i)} f[S,i,j] \cdot \mu[S,i] = \sum_{i=1}^\infty \sum_{j \in N(i)} f[S,j,i] \cdot \mu[S,i] \tag{$\star$}
\]
for all nonnegative functions $f$ on $\dGr$. Observe that $N(1) = \{2\}$, and
\[
N(i) = \{i - 1,i + 1,\wt{i - 1}\}
\]
where $\wt{i - 1}$ belongs to the orbit $\Aut(S)(i - 1)$ if $i \geq 2$. Let $f$ be a bounded nonnegative function on $\dGr$. Then
\[
\sum_{j \in N(i)} f[S,j,i] =
\begin{cases}
f[S,2,1] & \text{ if $i = 1$,}\\
2f[S,i - 1,i] + f[S,i + 1,i] & \text{ if $i \geq 2$}\\
\end{cases}
\]
where $f[S,\wt{i - 1},i] = f[S,i - 1,i]$ because there is an automorphism of $S$ that fixes $i$ and exchanges $i - 1$ and $\wt{i - 1}$. Similarly,
\[
\sum_{j \in N(i)} f[S,i,j] =
\begin{cases}
f[S,1,2] & \text{ if $i = 1$,}\\
2f[S,i,i - 1] + f[S,i,i + 1] & \text{ if $i \geq 2$.}\\
\end{cases}
\]
Since $f$ is bounded, the sums on either side of ($\star$) are finite, which means the equation ($\star$) is satisfied precisely when
\[
\sum_{i=1}^\infty \sum_{j \in N(i)} (f[S,j,i] - f[S,i,j]) \cdot \mu[S,i] = 0.
\]
Using these equations,
\begin{align*}
0 &= \sum_{i=1}^\infty \sum_{j \in N(i)} (f[S,j,i] - f[S,i,j]) \cdot \mu[S,i]\\
  &= \sum_{j \in N(1)} (f[S,j,1] - f[S,1,j]) \cdot \mu[S,1] + \sum_{i=2}^\infty \sum_{j \in N(i)} (f[S,j,i] - f[S,i,j]) \cdot \mu[S,i]\\
  &= (f[S,2,1] - f[S,1,2]) \cdot \mu[S,1]\\
  &+ (2f[S,1,2] + f[S,3,2] - 2f[S,2,1] - f[S,2,3]) \cdot \mu[S,2]\\
  &+ (2f[S,2,3] + f[S,4,3] - 2f[S,3,2] - f[S,3,4]) \cdot \mu[S,3] + \cdots
\end{align*}
because $\mu$ is unimodular. Note that $f[S,2,1]$ appears as a multiple of $\mu[S,1]$ and $\mu[S,2]$. Furthermore, $f[S,2,1]$ does not appear as any other multiple. The remaining terms possess a similar pattern. Let us rearrange, factor, and combine the terms in the summation above to obtain
\begin{align*}
0 &= f[S,2,1] \cdot (\mu[S,1] - 2\mu[S,2]) + f[S,1,2] \cdot (2\mu[S,2] - \mu[S,1])\\
  &+ f[S,3,2] \cdot (\mu[S,2] - 2\mu[S,3]) + f[S,2,3] \cdot (2\mu[S,3] - \mu[S,2]) + \cdots\\
  &= \sum_{i=1}^\infty (f[S,i + 1,i] - f[S,i,i + 1]) \cdot (\mu[S,i] - 2\mu[S,i + 1]). \tag{$\star\star$}
\end{align*}
By the definition of unimodularity, it is possible to let $f$ be the characteristic function of the singleton $\{[S,2,1]\}$. In this case, ($\star\star$) simplifies to the equation $0 = \mu[S,1] - 2\mu[S,2]$. That is, $\mu[S,1] = 2\mu[S,2]$. In fact, by choosing different characteristic functions, we see that $\mu[S,k] = 2\mu[S,k + 1]$ for all positive integers $k$.

These equations restrict the measure $\mu$, but they do not define it. However, $\mu$ is uniquely determined because $\mu(\Gr) = 1$. To see this, it suffices to calculate $\mu[S,1]$. By induction and the equation $\mu[S,k] = 2\mu[S,k + 1]$, it is easy to see that $\mu[S,1] = \mu[S,k] \cdot 2^{k-1}$ for all positive integers $k$. Then
\[
1 = \sum_{i=1}^\infty \mu[S,i] = \sum_{i=1}^\infty \frac{\mu[S,1]}{2^{i-1}} = 2\mu[S,1] \cdot \sum_{i=1}^\infty \frac{1}{2^i} = 2\mu[S,1],
\]
and so $\mu[S,1] = 1/2$. Since the measure of every rooted connected component of $S$ is a multiple of $\mu[S,1]$, it follows that $\mu$ is uniquely determined. Indeed, $\mu = \mu_S$, the previously mentioned measure.

%
%
\subsection{Example: the tree $T_{3,4}$ where $\Rcc(T_{3,4})$ is finite}

\begin{figure}[ht]
\begin{center}
\includegraphics[scale=1.2]{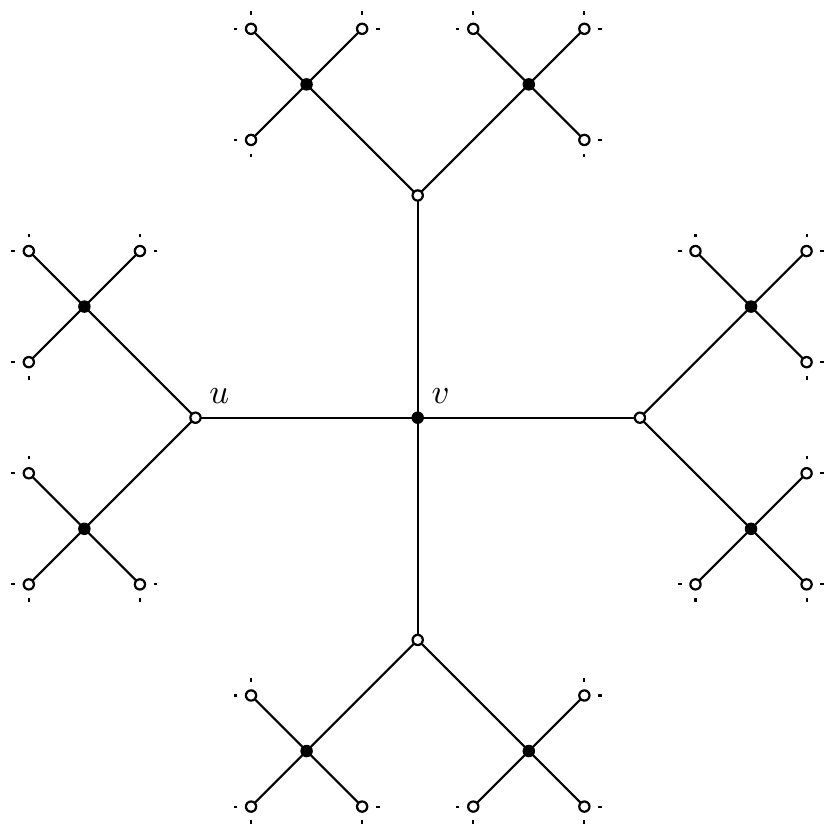}
\caption[The infinite tree $T_{3,4}$]{The infinite tree $X = T_{3,4}$; its orbits are $\Aut(X)u$ and $\Aut(X)v$, which are represented by unshaded and shaded vertices, respectively.}
\label{the_tree_t34}
\end{center}
\end{figure}

The following example deals with an infinite graph whose set of rooted connected components is finite. Consider the infinite tree $X = T_{3,4}$ depicted in Figure \ref{the_tree_t34} whose vertices alternate having $3$ and $4$ neighbours. Let $u$ be a vertex of degree $3$, and let $v$ be a vertex of degree $4$. The rooted connected components of $X$ are $[X,u]$ and $[X,v]$.

Suppose that $\mu$ is a unimodular measure sustained by $X$. That is, $\{[X,u],[X,v]\}$ has full measure. Note that $[X,u,y] = [X,u,v]$ for all $y \in N(u)$ and $[X,y,v] = [X,u,v]$ for all $y \in N(v)$. Denote by $f$ the characteristic function of the singleton $\{[X,u,v]\}$. Since $\mu$ is unimodular and $f$ is bounded,
\begin{align*}
0 &= \sum_{x \in \{u,v\}} \sum_{y \in N(x)} (f[X,x,y] - f[X,y,x]) \cdot \mu[X,x]\\
  &= \sum_{y \in N(u)} (f[X,u,y] - f[X,y,u]) \cdot \mu[X,u] + \sum_{y \in N(v)} (f[X,v,y] - f[X,y,v]) \cdot \mu[X,v]\\
  &= \sum_{y \in N(u)} f[X,u,y] \cdot \mu[X,u] - \sum_{y \in N(v)} f[X,y,v] \cdot \mu[X,v]\\
  &= |N(u)| f[X,u,v] \cdot \mu[X,u] - |N(v)| f[X,u,v] \cdot \mu[X,v]\\
  &= 3\mu[X,u] - 4\mu[X,v]
\end{align*}
where $f[X,y,u] = 0$ for all $y \in N(u)$ and $f[X,v,y] = 0$ for all $y \in N(v)$ because $[X,u] \neq [X,v]$. Furthermore, $\mu[X,u] + \mu[X,v] = 1$, and so these two equations yield $\mu[X,u] = 4/7$ and $\mu[X,v] = 3/7$. Hence $\mu$ is uniquely determined.

%
%
\subsection{Example: the tree $T_{3,2,4}$}

\begin{figure}[ht]
\begin{center}
\includegraphics[scale=1.2]{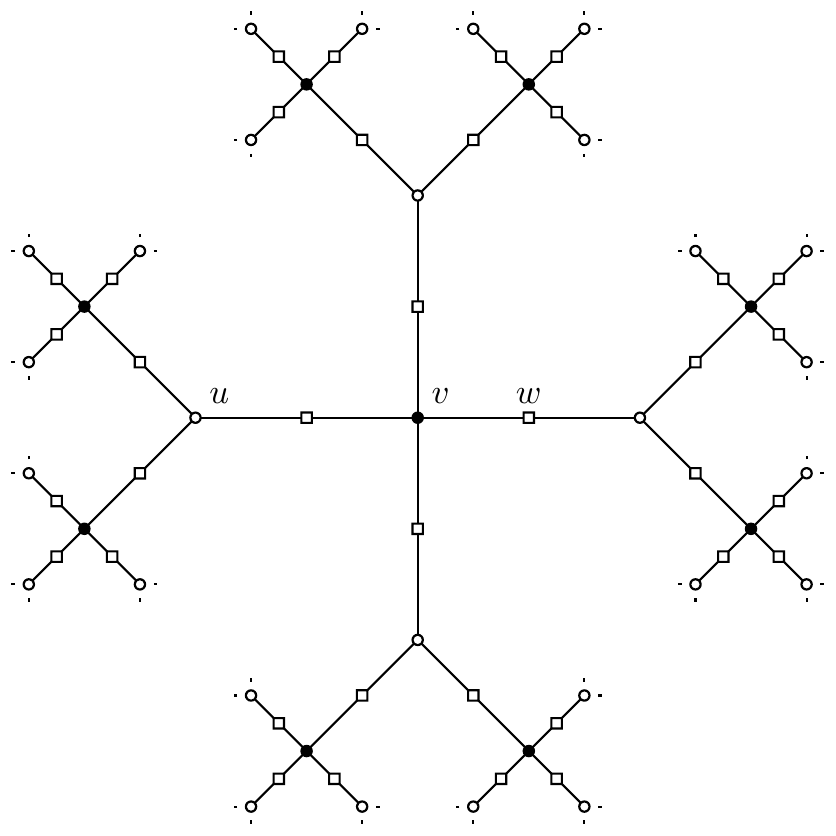}
\caption[The infinite tree $T_{3,2,4}$]{The infinite tree $Y = T_{3,2,4}$; its orbits are $\Aut(Y)u$, $\Aut(Y)v$, and $\Aut(Y)w$, which are represented by unshaded, shaded, and square vertices, respectively.}
\label{the_tree_t324}
\end{center}
\end{figure}

Let us build on the previous example by modifying the tree $T_{3,4}$. Consider the infinite tree $Y = T_{3,2,4}$, shown in Figure \ref{the_tree_t324}, defined by subdividing each edge of $T_{3,4}$ exactly once. This tree has three types of vertices: let $u$ be a vertex of degree $3$, let $v$ be a vertex of degree $4$, and let $w$ be a vertex of degree $2$. The set of rooted connected components of $Y$ is simply $\{[Y,u],[Y,v],[Y,w]\}$. Intuitively, the neighbours of $u$ belong to the orbit of $w$, the neighbours of $v$ belong to the orbit of $w$, and the neighbours of $w$ belong to both the orbits of $u$ and $v$.

Suppose that $\mu$ is a unimodular measure sustained by $Y$. Let $f$ be a bounded nonnegative function on $\dGr$. Observe that
\[
\sum_{y \in N(u)} (f[Y,u,y] - f[Y,y,u]) = 3(f[Y,u,w] - f[Y,w,u])
\]
and
\[
\sum_{y \in N(v)} (f[Y,v,y] - f[Y,y,v]) = 4(f[Y,v,w] - f[Y,w,v])
\]
using similar reasoning as in the previous example. Furthermore, $N(w) = \{u,v\}$ implies that
\[
\sum_{y \in N(w)} (f[Y,w,y] - f[Y,y,w]) = f[Y,w,u] - f[Y,u,w] + f[Y,w,v] - f[Y,v,w].
\]
Thus, for all bounded nonnegative functions $f$ on $\dGr$,
\begin{align*}
0 &= \sum_{x \in \{u,v,w\}} \sum_{y \in N(x)} (f[Y,x,y] - f[Y,y,x]) \cdot \mu[Y,x]\\
  &= 3(f[Y,u,w] - f[Y,w,u]) \cdot \mu[Y,u] + 4(f[Y,v,w] - f[Y,w,v]) \cdot \mu[Y,v]\\
  &+ (f[Y,w,u] - f[Y,u,w] + f[Y,w,v] - f[Y,v,w]) \cdot \mu[Y,w].
\end{align*}
In particular, this is true for the characteristic functions of the singletons $\{[Y,u,w]\}$ and $\{[Y,v,w]\}$. More precisely,
\[
0 = 3\mu[Y,u] + 0 + (-1)\mu[Y,w] = 3\mu[Y,u] - \mu[Y,w]
\]
and
\[
0 = 0 + 4\mu[Y,v] + (-1)\mu[Y,w] = 4\mu[Y,v] - \mu[Y,w],
\]
respectively. Since $\mu[Y,u] + \mu[Y,v] + \mu[Y,w] = 1$, we have a system of three equations,
\begin{empheq}[left=\empheqlbrace]{align*}
3\mu[Y,u] - \mu[Y,w]           &= 0\\
4\mu[Y,v] - \mu[Y,w]           &= 0\\
\mu[Y,u] + \mu[Y,v] + \mu[Y,w] &= 1
\end{empheq}
whose solution is $\mu[Y,u] = 4/19$, $\mu[Y,v] = 3/19$, and $\mu[Y,w] = 12/19$. Hence $\mu$ is uniquely determined.

%
%
\subsection{Connected graphs sustain at most one unimodular measure}

As these examples have shown, the claim that every connected graph sustains at most one unimodular measure is plausible. Having given the reader a taste of the ideas used in this section, it is time to demonstrate the main result.

For convenience, let $G = \Aut(X)$, and let $G_x = \{\varphi \in G ~:~ \varphi(x) = x\}$\index{Gx@$G_x$} be the \emph{stabilizer subgroup}\index{stabilizer subgroup} of $x \in V(X)$. With this notation, $Gx$\index{Gx@$Gx$} is the orbit of $x \in V(X)$ under the action of $G$, and $G_x y$\index{Gxy@$G_x y$} is the orbit of $y \in V(X)$ under the action of $G_x$. The following result is a new observation.

%
%
\begin{framed}
\begin{theo}\label{uniqueness_of_sustained_unimodular_measures}
Let $X$ be a connected graph. If $\mu$ and $\nu$ are unimodular measures sustained by $X$, then $\mu = \nu$.
\end{theo}
\end{framed}

\begin{rem}
Equivalently, whenever $X$ is a connected judicial graph, it sustains a unique unimodular measure. Denote this measure by $\Psi(X)$\index{PsiX@$\Psi(X)$!unique unimodular measure sustained by a connected judicial graph}.
\end{rem}

\begin{lem}\label{birooted_equiv_to_stabilizer}
Let $[X,a,b],[X,a,y],[X,x,b] \in \dGr$. Then $[X,a,b] = [X,a,y]$ if and only if $y \in G_a b$; and $[X,a,b] = [X,x,b]$ if and only if $x \in G_b a$.
\end{lem}

\begin{proof}
Suppose that $[X,a,b] = [X,a,y]$. Then there is a $\varphi \in G$ such that $\varphi(a) = a$ and $\varphi(b) = y$. Since $\varphi \in G_a$, it follows that $y \in G_a b$. On the other hand, assume that $y \in G_a b$. Then there is an automorphism $\varphi \in G_a$ such that $y = \varphi(b)$. Thus $[X,a,b] = [X,\varphi(a),\varphi(b)] = [X,a,y]$. Similar reasoning applies to the second statement.
\end{proof}

Proposition \ref{criterion_for_unimodularity} is a new criterion for unimodularity in the case of measures sustained by connected graphs. Aside from its importance in the proof of Theorem \ref{uniqueness_of_sustained_unimodular_measures}, it is very useful on its own, which is why it is not merely a lemma.

\begin{prop}\label{criterion_for_unimodularity}
Suppose that $\mu$ is a measure sustained by a connected graph $X$. Then $\mu$ is unimodular if and only if
\[
|G_a b|\mu[X,a] = |G_b a|\mu[X,b]
\]
for all adjacent vertices $a$ and $b$ of $X$. In addition, $\mu$ is unimodular only if
\[
|Gb \cap N(a)|\mu[X,a] = |Ga \cap N(b)|\mu[X,b]
\]
for all adjacent vertices $a$ and $b$ of $X$.
\end{prop}

\begin{proof}
Let $a$ and $b$ be adjacent vertices of $X$. Denote by $f$ the characteristic function of $\{[X,a,b]\}$. Observe that $f[X,x,y] = 0$ if $[X,x] \neq [X,a]$ or $[X,y] \neq [X,b]$ for all $[X,x,y] \in \dGr$. Then
\begin{align*}
\sum_{[X,x] \in \Rcc(X)} \sum_{y \in N(x)} f[X,x,y] \cdot \mu[X,x] &= \sum_{y \in N(a)} f[X,a,y] \cdot \mu[X,a]\\
                                                             &= \sum_{y \in G_a b} 1 \cdot \mu[X,a]\\
                                                             &= |G_a b|\mu[X,a]
\end{align*}
where $f[X,a,y] = 1$ if and only if $y \in G_a b$ by Lemma \ref{birooted_equiv_to_stabilizer} for all $y \in N(a)$. Similarly,
\[
\sum_{[X,x] \in \Rcc(X)} \sum_{y \in N(x)} f[X,y,x] \cdot \mu[X,x] = |G_b a|\mu[X,b].
\]
If $\mu$ is unimodular, it follows that $|G_a b|\mu[X,a] = |G_b a|\mu[X,b]$.

Conversely, assume that $|G_a b|\mu[X,a] = |G_b a|\mu[X,b]$ for all adjacent vertices $a$ and $b$ of $X$. Then
\[
\sum_{[X,x] \in \Rcc(X)} \sum_{y \in N(x)} \chi_{(a,b)}[X,x,y] \cdot \mu[X,x] = \sum_{[X,x] \in \Rcc(X)} \sum_{y \in N(x)} \chi_{(a,b)}[X,y,x] \cdot \mu[X,x] \tag{$\star$}
\]
holds where $\chi_{(a,b)}$ is the characteristic function of the singleton $\{[X,a,b]\}$ for all adjacent vertices $a$ and $b$. To see that $\mu$ is unimodular, let $f$ be a nonnegative function on $\dGr$ with $f(\dGr \setminus \BRcc(X)) = \{0\}$. It is easy to verify that
\[
f = \sum_{[X,a,b] \in \BRcc(X)} f[X,a,b] \cdot \chi_{(a,b)}.
\]
Since $X$ is countable, so are $\Rcc(X)$ and $\BRcc(X)$. Then
\begin{align*}
\sum_{[X,x] \in \Rcc(X)} &\sum_{y \in N(x)} f[X,x,y] \cdot \mu[X,x]\\
                         &= \sum_{[X,x] \in \Rcc(X)} \sum_{y \in N(x)} \left(\sum_{[X,a,b] \in \BRcc(X)} f[X,a,b] \cdot \chi_{(a,b)}[X,x,y]\right) \mu[X,x]\\
                         &= \sum_{[X,a,b] \in \BRcc(X)} f[X,a,b] \left(\sum_{[X,x] \in \Rcc(X)} \sum_{y \in N(x)} \chi_{(a,b)}[X,x,y] \cdot \mu[X,x]\right)
\end{align*}
where interchanging the summations is possible because the terms are nonnegative. Similarly,
\begin{align*}
\sum_{[X,x] \in \Rcc(X)} &\sum_{y \in N(x)} f[X,y,x] \cdot \mu[X,x]\\
                         &= \sum_{[X,a,b] \in \BRcc(X)} f[X,a,b] \left(\sum_{[X,x] \in \Rcc(X)} \sum_{y \in N(x)} \chi_{(a,b)}[X,y,x] \cdot \mu[X,x]\right).
\end{align*}
By ($\star$), it follows that
\[
\sum_{[X,x] \in \Rcc(X)} \sum_{y \in N(x)} f[X,x,y] \cdot \mu[X,x] = \sum_{[X,x] \in \Rcc(X)} \sum_{y \in N(x)} f[X,y,x] \cdot \mu[X,x],
\]
and so $\mu$ is unimodular, as required.

To demonstrate the second part of the proposition, let $a$ and $b$ be adjacent vertices of $X$. Denote by $f$ the characteristic function of the set
\[
\{[X,x,y] \in \dGr ~:~ x \in Ga ~\text{ and }~ y \in Gb\}.
\]
Observe that $f[X,x,y] = 0$ if $[X,x] \neq [X,a]$ or $[X,y] \neq [X,b]$ for all $[X,x,y] \in \dGr$. Furthermore, $f[X,a,y] = 1$ if and only if $y \in Gb \cap N(a)$. Then
\begin{align*}
\sum_{[X,x] \in \Rcc(X)} \sum_{y \in N(x)} f[X,x,y] \cdot \mu[X,x] &= \sum_{y \in N(a)} f[X,a,y] \cdot \mu[X,a]\\
                                                             &= \sum_{y \in Gb \cap N(a)} 1 \cdot \mu[X,a]\\
                                                             &= |Gb \cap N(a)|\mu[X,a].
\end{align*}
Similarly,
\[
\sum_{[X,x] \in \Rcc(X)} \sum_{y \in N(x)} f[X,y,x] \cdot \mu[X,x] = |Ga \cap N(b)|\mu[X,b].
\]
If $\mu$ is unimodular, it follows that $|Gb \cap N(a)|\mu[X,a] = |Ga \cap N(b)|\mu[X,b]$.
\end{proof}

\begin{proof}[Proof of Theorem \ref{uniqueness_of_sustained_unimodular_measures}]
By Proposition \ref{criterion_for_unimodularity},
\[
|G_x y|\mu[X,x] = |G_y x|\mu[X,y]
\]
and
\[
|G_x y|\nu[X,x] = |G_y x|\nu[X,y]
\]
for all adjacent vertices $x$ and $y$ of $X$. Let $[X,1] \in \Rcc(X)$ be a rooted connected component of $X$, and let $x \in V(X)$. Since $X$ is connected, there is a path $(x_0,x_1,\ldots,x_k)$ with $x_0 = 1$ and $x_k = x$ in $X$. Then
\[
\frac{\mu[X,x_{i+1}]}{\mu[X,x_i]} = \frac{\nu[X,x_{i+1}]}{\nu[X,x_i]}
\]
for all $i \in \{0,1,\ldots,k - 1\}$. The product of these terms is
\[
\frac{\mu[X,x_1]}{\mu[X,x_0]} \frac{\mu[X,x_2]}{\mu[X,x_1]} \cdots \frac{\mu[X,x_k]}{\mu[X,x_{k-1}]} = \frac{\nu[X,x_1]}{\nu[X,x_0]} \frac{\nu[X,x_2]}{\nu[X,x_1]} \cdots \frac{\nu[X,x_k]}{\nu[X,x_{k-1}]}.
\]
That is,
\[
\frac{\mu[X,x_k]}{\mu[X,x_0]} = \frac{\nu[X,x_k]}{\nu[X,x_0]},
\]
which means $\mu[X,x]\nu[X,1] = \nu[X,x]\mu[X,1]$. It follows that
\[
\mu[X,1] = \left(\sum_{[X,x] \in \Rcc(X)} \nu[X,x]\right) \mu[X,1] = \left(\sum_{[X,x] \in \Rcc(X)} \mu[X,x]\right) \nu[X,1] = \nu[X,1]
\]
because $\mu$ and $\nu$ are probability measures sustained by $X$. Hence $\mu[X,1] = \nu[X,1]$, and so $\mu[X,x] = \nu[X,x]$.
\end{proof}

%
%
\subsection{Existence of a sustained unimodular measure}

Recall that $G = \Aut(X)$ and $G_x$ is the stabilizer subgroup of $x \in V(X)$ where $X$ is a graph. So far, we have shown that a connected graph sustains at most one unimodular measure. It is natural to ask under what conditions such a measure exists. According to Proposition \ref{criterion_for_unimodularity}, a sustained unimodular measure must satisfy the equation $|G_a b|\mu[X,a] = |G_b a|\mu[X,b]$ for all adjacent vertices $a$ and $b$ of $X$. Based on this, we see that the measure of $[X,a]$ is just a multiple of the measure of $[X,b]$, which leads us to the following observation.

%
%
\begin{framed}
\begin{theo}\label{existence_of_sustained_unimodular_measures}
Let $X$ be a connected graph; let $[X,1] \in \Rcc(X)$; let $p$ be a positive real number. Define a function $\mu : \Rcc(X) \to \RR$ as follows: $\mu[X,1] = p$, and for all $[X,x] \in \Rcc(X) \setminus \{[X,1]\}$,
\[
\mu[X,x] = \frac{|G_{x_0}x_1| |G_{x_1}x_2| \cdots |G_{x_{k-1}}x_k|}{|G_{x_1}x_0| |G_{x_2}x_1| \cdots |G_{x_k}x_{k-1}|} p
\]
where $(x_0,x_1,\ldots,x_k)$ is a path in $X$ with $x_0 = 1$ and $x_k = x$. Then the following statements hold.

\begin{enumerate}
\item The function $\mu$ is independent of the choice of path.

\item If $\mu(\Rcc(X)) = 1$, then $\mu$ is the unique unimodular measure sustained by $X$.
\end{enumerate}
\end{theo}
\end{framed}

The validity of Theorem \ref{existence_of_sustained_unimodular_measures} rests entirely on the well-definedness of the function $\mu$. Indeed, there may be several distinct paths between two vertices. However, the following lemmas are sufficient to see that $\mu$ must be well-defined. The idea is to use the Haar measure on $G$, as suggested in a preprint by Vadim Kaimanovich \cite{kaimanovich13}.

\begin{lem}\label{ratio_of_haar_measures}
Let $G = \Aut(X)$ for some connected graph $X$. Denote by $m$ the left Haar measure on $G$. If $a$ and $b$ are adjacent vertices of $X$, then $|G_a b| \cdot m(G_a \cap G_b) = m(G_a)$.
\end{lem}

\begin{proof}
Let $a$ and $b$ be adjacent vertices of $X$. It is known that the stabilizers $G_a$ and $G_b$ are compact subgroups of $G$, as shown, for example, by Woess \cite{woess91}.

Define the equivalence relation $\sim$ on $G_a$ as follows: $\alpha \sim \beta$ if $\alpha(b) = \beta(b)$ for all $\alpha, \beta \in G_a$. Clearly, $[1] = G_a \cap G_b$ where $1 \in G$ is the identity map, and $[1]$ is its equivalence class. In addition, it is easy to see that $\alpha \cdot [1] = [\alpha]$ for all $\alpha \in G_a$. Furthermore, the function $f : G_a/{\sim} \to G_a b$ defined on the quotient set by $f[\alpha] = \alpha(b)$ is a well-defined bijection, which means $|G_a/{\sim}| = |G_a b|$. Using the left-translation-invariance and additivity of the Haar measure $m$, it follows that
\[
m(G_a) = m\left(\bigcup_{[\alpha] \in G_a/{\sim}} \alpha \cdot [1]\right) = \sum_{[\alpha] \in G_a/{\sim}} m(\alpha \cdot (G_a \cap G_b)) = |G_a b| \cdot m(G_a \cap G_b)
\]
where the second equality holds because the union is disjoint.
\end{proof}

\begin{lem}\label{walk_equiv_reverse_walk}
Let $X$ be a connected graph. Suppose that $(x_0,x_1,\ldots,x_k,x_0)$ is a walk in $X$. Then
\[
|G_{x_0}x_1| |G_{x_1}x_2| \cdots |G_{x_{k-1}}x_k| |G_{x_k}x_0| = |G_{x_1}x_0| |G_{x_2}x_1| \cdots |G_{x_k}x_{k-1}| |G_{x_0}x_k|.
\]
\end{lem}

\begin{proof}
According to Lemma \ref{ratio_of_haar_measures},
\[
\frac{|G_a b|}{|G_b a|} = \frac{m(G_a)}{m(G_a \cap G_b)} \cdot \frac{m(G_b \cap G_a)}{m(G_b)} = \frac{m(G_a)}{m(G_b)}
\]
for all adjacent vertices $a$ and $b$ of $X$ where $m$ is the left Haar measure on $G$, and so
\[
\frac{|G_{x_0}x_1|}{|G_{x_1}x_0|} \frac{|G_{x_1}x_2|}{|G_{x_2}x_1|} \cdots \frac{|G_{x_{k-1}}x_k|}{|G_{x_k}x_{k-1}|} \frac{|G_{x_k}x_0|}{|G_{x_0}x_k|} = \frac{m(G_{x_0})}{m(G_{x_1})} \frac{m(G_{x_1})}{m(G_{x_2})} \cdots \frac{m(G_{x_{k-1}})}{m(G_{x_k})} \frac{m(G_{x_k})}{m(G_{x_0})} = 1,
\]
as required.
\end{proof}

\begin{figure}[ht]
\begin{center}
\includegraphics[scale=1.2]{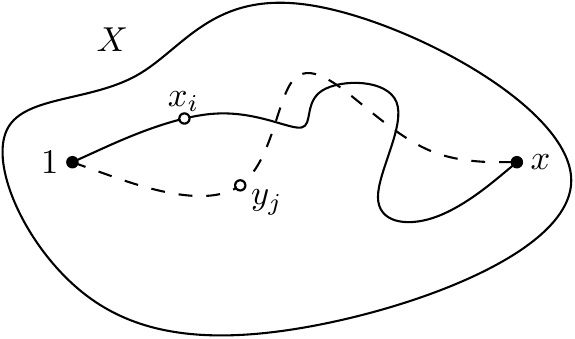}
\caption[The construction of a unimodular measure is independent of the path]{The construction of a unimodular measure is independent of the path; $(x_0,x_1,\ldots,x_k)$ and $(y_0,y_1,\ldots,y_l)$ are represented by a solid line and a dashed line, respectively.}
\label{independent_of_path}
\end{center}
\end{figure}

\begin{proof}[Proof of Theorem \ref{existence_of_sustained_unimodular_measures}]
Let $X$ be a connected graph, and let $[X,1] \in \Rcc(X)$. To see that the function $\mu$ is independent of the choice of path, assume that $(y_0,y_1,\ldots,y_l)$ is another path in $X$ with $y_0 = 1$ and $y_l = x$. As shown in Figure \ref{independent_of_path},
\[
(x_0,x_1,\ldots,x_k,y_{l-1},\ldots,y_1,y_0)
\]
is a walk in $X$ with $x_0 = y_0$ and $x_k = y_l$. By Lemma \ref{walk_equiv_reverse_walk}, it is easy to deduce that
\[
\frac{|G_{x_0}x_1| |G_{x_1}x_2| \cdots |G_{x_{k-1}}x_k|}{|G_{x_1}x_0| |G_{x_2}x_1| \cdots |G_{x_k}x_{k-1}|} = \frac{|G_{y_0}y_1| |G_{y_1}y_2| \cdots |G_{y_{l-1}}y_l|}{|G_{y_1}y_0| |G_{y_2}y_1| \cdots |G_{y_l}y_{l-1}|},
\]
which means $\mu$ is well-defined.

It remains to demonstrate that $\mu$ is unimodular. Suppose that $a$ and $b$ are adjacent vertices of $X$. Without loss of generality, assume that
\[
\mu[X,a] = \frac{|G_{x_0}x_1| |G_{x_1}x_2| \cdots |G_{x_k} a|}{|G_{x_1}x_0| |G_{x_2}x_1| \cdots |G_a x_k|} p
\]
and
\[
\mu[X,b] = \frac{|G_{x_0}x_1| |G_{x_1}x_2| \cdots |G_{x_k} a| |G_a b|}{|G_{x_1}x_0| |G_{x_2}x_1| \cdots |G_a x_k| |G_b a|} p
\]
where $(x_0,x_1,\ldots,x_k,a)$ and $(x_0,x_1,\ldots,x_k,a,b)$ are paths in $X$ with $x_0 = 1$. Thus $|G_b a|\mu[X,b] = |G_a b|\mu[X,a]$, and so $\mu$ is unimodular according to Proposition \ref{criterion_for_unimodularity}.
\end{proof}

%
%
\subsection{The barred binary tree is lawless}

\begin{figure}[ht]
\begin{center}
\includegraphics[scale=1.2]{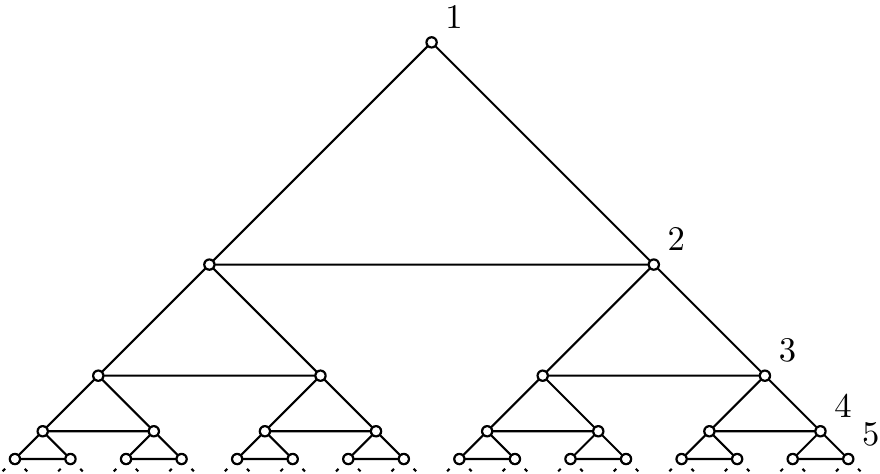}
\caption[The barred binary tree]{The barred binary tree; its orbits are labelled using the positive integers.}
\label{labelled_barred_binary_tree}
\end{center}
\end{figure}

The following example should convince the reader that a connected graph need not sustain a unimodular measure, even if it is not rigid. Consider the barred binary tree $X = \bar\Lambda$. For the benefit of the reader, its orbits are labelled using the positive integers, and the resulting tree is shown in Figure \ref{labelled_barred_binary_tree}. Note that $|G_k (k + 1)| = 2$ and $|G_{k+1} k| = 1$ for all positive integers $k$. Suppose that $X$ sustains a unimodular measure $\mu$. Then Proposition \ref{criterion_for_unimodularity} tells us that the equation
\[
2\mu[X,k] - \mu[X,k + 1] = 0
\]
is true for all positive integers $k$. It follows that $\mu[X,k] = 2^{k-1} \mu[X,1]$, and so
\[
\sum_{k=1}^\infty \mu[X,k] = \mu[X,1] \cdot \sum_{k=1}^\infty 2^{k-1}
\]
where the right-hand side is unbounded unless $\mu[X,1] = 0$. In either case, $\mu(\Gr) \neq 1$, which is a contradiction.

%
%
\subsection{Applications}

To demonstrate the usefulness of Proposition \ref{criterion_for_unimodularity}, our so-called unimodularity criterion, we discuss several results that apply this fact. To begin, it is possible to extend Proposition \ref{infinite_connected_rigid_graphs_not_judicial}, which states that infinite connected rigid graphs are lawless.

\begin{lem}\label{orbit_neighbourhood_ratio}
Suppose that $X$ is a connected graph whose orbits under the action of $G = \Aut(X)$ are all finite. Then $|Gx||Gy \cap N(x)| = |Gy||Gx \cap N(y)|$ for all adjacent vertices $x$ and $y$ of $X$.
\end{lem}

\begin{proof}
If $Gx = Gy$, then the result holds. On the other hand, assume that $Gx \neq Gy$. Let $Gx = \{x_1,x_2,\ldots,x_k\}$, and let $Gy = \{y_1,y_2,\ldots,y_l\}$ with $x_1 = x$ and $y_1 = y$. Clearly,
\[
|Gy \cap N(x)| = |Gy \cap N(x_i)|
\]
for all $i \in \{1,2,\ldots,k\}$, and
\[
|Gx \cap N(y)| = |Gx \cap N(y_j)|
\]
for all $j \in \{1,2,\ldots,l\}$. Then the number of edges $\{a,b\}$ of $X$ where $a \in Gx$ and $b \in Gy$ may be counted in two ways:
\[
\sum_{i=1}^k |Gy \cap N(x_i)| = \sum_{i=1}^k |Gy \cap N(x)| = |Gx||Gy \cap N(x)|
\]
and
\[
\sum_{j=1}^l |Gx \cap N(y_j)| = \sum_{j=1}^l |Gx \cap N(y)| = |Gy||Gx \cap N(y)|,
\]
which means $|Gx||Gy \cap N(x)| = |Gy||Gx \cap N(y)|$.
\end{proof}

Lemma \ref{orbit_neighbourhood_ratio} and Proposition \ref{criterion_for_unimodularity} yield a straightforward proof of the following new observation, which extends Proposition \ref{infinite_connected_rigid_graphs_not_judicial}.

\begin{theo}
If $X$ is an infinite connected graph whose orbits are all finite, then there is no unimodular measure sustained by $X$.
\end{theo}

\begin{proof}
Suppose that $\mu$ is a unimodular measure sustained by $X$. Let $[X,1] \in \Rcc(X)$. By Proposition \ref{criterion_for_unimodularity} and Lemma \ref{orbit_neighbourhood_ratio},
\[
\frac{\mu[X,a]}{\mu[X,b]} = \frac{|Ga \cap N(b)|}{|Gb \cap N(a)|} = \frac{|Ga|}{|Gb|}
\]
for all adjacent vertices $a$ and $b$ of $X$. Let $[X,x] \in \Rcc(X)$, and let $(x_0,x_1,\ldots,x_k)$ be a path in $X$ with $x_0 = 1$ and $x_k = x$. Then
\[
\frac{\mu[X,x]}{\mu[X,1]} = \frac{\mu[X,x_1]}{\mu[X,x_0]} \frac{\mu[X,x_2]}{\mu[X,x_1]} \cdots \frac{\mu[X,x_k]}{\mu[X,x_{k-1}]} = \frac{|Gx_1|}{|Gx_0|} \frac{|Gx_2|}{|Gx_1|} \cdots \frac{|Gx_k|}{|Gx_{k-1}|} = \frac{|Gx|}{|G1|},
\]
and so $|G1|\mu[X,x] = |Gx|\mu[X,1]$. It follows that
\[
|G1| = \sum_{[X,x] \in \Rcc(X)} |G1|\mu[X,x] = \sum_{[X,x] \in \Rcc(X)} |Gx|\mu[X,1] = |V(X)|\mu[X,1]
\]
because $\mu$ is a probability measure. However, this is a contradiction. Hence there is no unimodular measure sustained by $X$.
\end{proof}

In the results below, we generalize two facts that were proven in this author's Honours project \cite{artemenko11a}.

\begin{prop}\label{sustained_implies_strictly_sustained}
If $\mu$ is a unimodular measure sustained by a connected graph $X$, then $\mu$ is strictly sustained by $X$.
\end{prop}

\begin{proof}
Suppose that $\mu$ is a unimodular measure sustained by a connected graph $X$. To derive a contradiction, assume that $\mu$ is not strictly sustained by $X$. That is, $\mu[X,x] = 0$ for some $x \in V(X)$. Let $y \in N(x)$. By Proposition \ref{criterion_for_unimodularity},
\[
|G_x y|\mu[X,x] = |G_y x|\mu[X,y],
\]
and so $\mu[X,y] = 0$ because $|G_y x| \neq 0$. Since $X$ is connected, it follows that $\mu(\Rcc(X)) = 0$; a contradiction.
\end{proof}

\begin{cor}
Let $X$ and $Y$ be connected graphs that sustain unimodular measures $\mu$ and $\nu$, respectively. If $\mu = \nu$, then $X = Y$.
\end{cor}

\begin{proof}
Suppose that $\mu = \nu$. For all $[Y,y] \in \Rcc(Y)$, $\mu[Y,y] = \nu[Y,y] \neq 0$ by Proposition \ref{sustained_implies_strictly_sustained}, and so $[Y,y] \in \Rcc(X)$. Similarly, $\Rcc(X) \subseteq \Rcc(Y)$. It follows that the rooted connected components of $X$ are precisely those of $Y$. Hence $X = Y$.
\end{proof}

%
%
\subsubsection{Cayley graphs}

The criterion for unimodularity presented in Proposition \ref{criterion_for_unimodularity} is truly useful in its own right. It provides a simple method to determine whether a connected vertex-transitive graph is judicial. As previously stated, let $G = \Aut(X)$.

\begin{cor}\label{unimodularity_for_vtransitive_graphs}
A connected vertex-transitive graph $X$ is judicial if and only if $|G_a b| = |G_b a|$ for all adjacent vertices $a$ and $b$ of $X$.
\end{cor}

\begin{proof}
Since $X$ is vertex-transitive, it has a unique rooted connected component $[X,\cdot]$. Thus the only measure sustained by $X$ is the Dirac measure $\delta_X$ defined by $\delta_X[X,\cdot] = 1$. By Proposition \ref{criterion_for_unimodularity}, $\delta_X$ is unimodular if and only if
\[
|G_a b| = |G_a b|\delta_X[X,a] = |G_b a|\delta_X[X,b] = |G_b a|
\]
for all adjacent vertices $a$ and $b$ of $X$.
\end{proof}

This author is curious to know whether there is a similar ``measure independent'' characterization of more general judicial graphs.

Unfortunately, even the seemingly weak condition in Corollary \ref{unimodularity_for_vtransitive_graphs} is not satisfied for all vertex-transitive graphs. For a counterexample, the reader is encouraged to study the so-called \emph{grandparent graph}, which is constructed in the text by Russell Lyons and Yuval Peres \cite[p.~257]{lyonsperes13}. However, as the same reference mentions, there is a prominent subcollection of connected vertex-transitive graphs, each of which is judicial.

\begin{defn}
Let $\Gamma$ be a group, and let $S$ be a generating set of $\Gamma$ that does not contain the identity $e$ with $S = S^{-1}$. The \emph{Cayley graph $X = \Cay(\Gamma,S)$ of $\Gamma$}\index{Cayley graph} is the graph defined as follows: $V(X) = \Gamma$ and $E(X) = \{\{g,gs\} ~:~ g \in \Gamma ~\text{ and }~ s \in S\}$.
\end{defn}

Note that a Cayley graph $X$ is connected because it is constructed using a set that generates the entire group. Since the edges of Cayley graphs have a specific form, it is possible to rephrase Corollary \ref{unimodularity_for_vtransitive_graphs} as follows. Suppose that $X = \Cay(\Gamma,S)$ is a Cayley graph for some group $\Gamma$. Then $X$ is judicial if and only if $|G_g (gs)| = |G_{(gs)} g|$ for all $g \in \Gamma$ and $s \in S$. Thus it suffices to demonstrate the following result.

\begin{prop}\label{cay_graph_stabs_are_equal}
If $X = \Cay(\Gamma,S)$ is a Cayley graph for some group $\Gamma$, then $|G_g (gs)| = |G_{(gs)} g|$ for all $g \in \Gamma$ and $s \in S$.
\end{prop}

\begin{lem}\label{lower_s_and_upper_s_functions}
Under the assumptions of Proposition \ref{cay_graph_stabs_are_equal}, let $g \in \Gamma$, let $s \in S$, and let $\alpha$ and $\beta$ be automorphisms of $X$. Define $\alpha_s : X \to X$ and $\beta^s : X \to X$ by
\[
\alpha_s(h) = gs(\alpha(gs)^{-1})\alpha(h)
\]
and
\[
\beta^s(h) = g(\beta(g)^{-1})\beta(h)
\]
for all $h \in \Gamma$. Then (i) $\alpha_s \in G_{(gs)}$ and (ii) $\beta^s \in G_g$. Furthermore, assume that $\alpha \in G_g$ and $\beta \in G_{(gs)}$. Then $(\alpha_s)^s = \alpha$ and $(\beta^s)_s = \beta$.
\end{lem}

\begin{proof}
In either case, we must prove that $\alpha_s$ and $\beta^s$ are injective, surjective, preserve adjacency, and satisfy the equations $\alpha_s(gs) = gs$ and $\beta^s(g) = g$.

Let us begin by showing that (i) is true. To see that $\alpha_s$ is injective, assume that $h_1$ and $h_2$ are vertices of $X$. Then
\[
\alpha_s(h_1) = \alpha_s(h_2) ~~\Rightarrow~~ gs(\alpha(gs)^{-1})\alpha(h_1) = gs(\alpha(gs)^{-1})\alpha(h_2) ~~\Rightarrow~~ \alpha(h_1) = \alpha(h_2),
\]
which means $h_1 = h_2$ because $\alpha$ is injective. Let $k \in \Gamma$. Since $\alpha$ is surjective, there is an $h \in \Gamma$ such that $\alpha(h) = \alpha(gs)s^{-1}g^{-1}k$. Then
\[
\alpha_s(h) = gs(\alpha(gs)^{-1})\alpha(h) = gs(\alpha(gs)^{-1})\alpha(gs)s^{-1}g^{-1}k = k,
\]
and so $\alpha_s$ is surjective too. The following argument shows that $\alpha_s$ preserves adjacency. Suppose that $h_1$ and $h_2$ are vertices of $X$. Then $h_1$ and $h_2$ are adjacent precisely when $\alpha(h_1)$ and $\alpha(h_2)$ are. Furthermore,
\begin{align*}
\alpha(h_2) = \alpha(h_1)s' ~~&\Leftrightarrow~~ gs(\alpha(gs)^{-1})\alpha(h_2) = gs(\alpha(gs)^{-1})\alpha(h_1)s'\\
                              &\Leftrightarrow~~ \alpha_s(h_2) = \alpha_s(h_1)s'
\end{align*}
for all $s' \in S$. It follows that $\alpha(h_1)$ and $\alpha(h_2)$ are adjacent if and only if $\alpha_s(h_1)$ and $\alpha_s(h_2)$ are too. Lastly, $\alpha_s(gs) = gs(\alpha(gs)^{-1})\alpha(gs) = gs$. Hence $\alpha_s \in G_{(gs)}$. The proof of (ii) is analogous, and so it is omitted.

The second part of the lemma involves two simple calculations, which are aided by the assumption that $\alpha(g) = g$ and $\beta(gs) = gs$:
\begin{align*}
(\alpha_s)^s(h) &= g(\alpha_s(g)^{-1})\alpha_s(h)\\
                &= g(gs(\alpha(gs)^{-1})\alpha(g))^{-1}gs(\alpha(gs)^{-1})\alpha(h)\\
                &= g(\alpha(g)^{-1})\alpha(gs)s^{-1}g^{-1}gs(\alpha(gs)^{-1})\alpha(h)\\
                &= g(\alpha(g)^{-1})\alpha(h)\\
                &= \alpha(h)
\end{align*}
and
\begin{align*}
(\beta^s)_s(h) &= gs(\beta^s(gs)^{-1})\beta^s(h)\\
               &= gs(g(\beta(g)^{-1})\beta(gs))^{-1}g(\beta(g)^{-1})\beta(h)\\
               &= gs(\beta(gs)^{-1})\beta(g)g^{-1}g(\beta(g)^{-1})\beta(h)\\
               &= gs(\beta(gs)^{-1})\beta(h)\\
               &= \beta(h)
\end{align*}
for all $h \in \Gamma$. Thus $(\alpha_s)^s = \alpha$ and $(\beta^s)_s = \beta$.
\end{proof}

\begin{proof}[Proof of Proposition \ref{cay_graph_stabs_are_equal}]
Let $g \in \Gamma$, and let $s \in S$. To see that $|G_g (gs)| = |G_{(gs)} g|$, it suffices to construct a bijection between $G_g (gs)$ and $G_{(gs)} g$. Define $f : G_g (gs) \to G_{(gs)} g$ by $f(\alpha(gs)) = \alpha_s(g)$ for all $\alpha(gs) \in G_g (gs)$ where $\alpha \in G_g$. Consider the function $F : G_{(gs)} g \to G_g (gs)$ defined by $F(\beta(g)) = \beta^s(gs)$ for all $\beta(g) \in G_{(gs)} g$ where $\beta \in G_{(gs)}$. Lemma \ref{lower_s_and_upper_s_functions} tells us that $f$ and $F$ are well-defined. In fact, as the following argument demonstrates, $F$ is the inverse of $f$. Let $\alpha(gs) \in G_g (gs)$ for some $\alpha \in G_g$, and let $\beta(g) \in G_{(gs)} g$ for some $\beta \in G_{(gs)}$. Then
\[
(F \circ f)(\alpha(gs)) = F(f(\alpha(gs))) = F(\alpha_s(g)) = (\alpha_s)^s(gs) = \alpha(gs)
\]
and
\[
(f \circ F)(\beta(g)) = f(F(\beta(g))) = f(\beta^s(gs)) = (\beta^s)_s(g) = \beta(g)
\]
where $(\alpha_s)^s(gs) = \alpha(gs)$ and $(\beta^s)_s(g) = \beta(g)$ by Lemma \ref{lower_s_and_upper_s_functions}. Hence there is a bijection between $G_g (gs)$ and $G_{(gs)} g$, which means these two sets are of equal size.
\end{proof}

The inspiration for the following result came from the known fact that Cayley graphs are unimodular in the sense of group unimodularity \cite[p.~8]{aldouslyons07}, which concerns the right invariance of the left Haar measure. Although this fact is demonstrated by Russell Lyons and Yuval Peres \cite[p.~305]{lyonsperes13}, we restate the result in our language of unimodularity.

\begin{cor}
Every Cayley graph is judicial.\qed
\end{cor}

%
%
\subsubsection{Convexity}
Denote by $\ol{\M}$\index{M@$\ol{\M}$} the closure of the set of laws $\M$\index{M@$\M$} of finite graphs. The remainder of this section is devoted to the study of the convexity of $\ol{\M}$. As demonstrated in the author's Honours project \cite{artemenko11a}, $\ol{\M}$ is convex. Furthermore, it is known that $\ol{\M}$ is weakly compact \cite{schramm08}, which means it is reasonable to determine its extreme points. Since the set of unimodular measures $\U$ is convex too, we attempt to calculate its extreme points as well.

\begin{defn}
An \emph{extreme point}\index{extreme point} of a convex set $A$ is an element $x \in A$ such that if $x = (y + z)/2$ for some $y,z \in A$, then $y = x$ and $z = x$.
\end{defn}

\begin{lem}\label{sustained_by_same_graph}
Let $\mu$ be a measure sustained by a connected graph $X$. If $\mu = (\eta + \nu)/2$ for some measures $\eta$ and $\nu$ on $\Gr$, then $\eta$ and $\nu$ are sustained by $X$.
\end{lem}

\begin{proof}
Suppose that $\mu = (\eta + \nu)/2$ for some measures $\eta$ and $\nu$ on $\Gr$. Without loss of generality, assume that $\eta[Y,y] > 0$ for some $[Y,y] \notin \Rcc(X)$. Then $\mu[Y,y] = 0$, but $(\eta[Y,y] + \nu[Y,y])/2 > 0$; a contradiction. It follows that $\eta[Y,y] = 0$ for all $[Y,y] \notin \Rcc(X)$. Thus $\eta$ must be sustained by $X$. Similarly, $\nu$ is sustained by $X$.
\end{proof}

\begin{prop}
If $\mu \in \ol{\M}$ is sustained by a connected graph, then $\mu$ is an extreme point of $\ol{\M}$.
\end{prop}

\begin{proof}
Suppose that $\mu \in \ol{\M}$ is sustained by a connected graph $X$, and assume that $\mu = (\eta + \nu)/2$ for some $\eta,\nu \in \ol{\M}$. By Lemma \ref{sustained_by_same_graph}, $\eta$ and $\nu$ are sustained by $X$ too. Then Theorem \ref{uniqueness_of_sustained_unimodular_measures} tells us that $\mu = \eta = \nu$. Hence $\mu$ is an extreme point of $\ol{\M}$.
\end{proof}

A nearly identical argument yields the following proposition, which informs the reader of some of the extreme points of the set of unimodular measures.

\begin{prop}\label{unimodular_sustained_by_connected_is_extreme_pt}
If $\mu$ is a unimodular measure sustained by a connected graph, then $\mu$ is an extreme point of the convex set $\U$ of unimodular measures.\qed
\end{prop}

The converse of Proposition \ref{unimodular_sustained_by_connected_is_extreme_pt} states that every extreme point of $\U$ is a unimodular measure sustained by a connected graph. Whether or not it is true is unknown to this author.

%% file: judiciality_of_disconnected_graphs.tex
%
%
\section{Judiciality of disconnected graphs}

Recall that a graph is \emph{judicial} if it sustains a unimodular measure, and it is \emph{lawless} otherwise. Consider a graph whose connected components are judicial. In this section, we demonstrate that such a graph itself is judicial, even if it has countably many components.

As previously shown, if $X$ is a connected judicial graph, it sustains a unique unimodular measure $\Psi(X)$. However, this does not hold if the graph is not connected. Indeed, let $Y = 2K_1 + K_2$ where $K_1$ and $K_2$ are the complete graphs on $1$ and $2$ vertices, respectively. Of course, the law $\Psi(Y)$ is an example of a unimodular measure sustained by $Y$. However, the measure $\mu$ defined by $\mu[K_1,\cdot] = 1$ and $\mu[K_2,\cdot] = 0$ is unimodular too; here $[K_1,\cdot]$ and $[K_2,\cdot]$ are the rooted connected components of $Y$. It is easy to see that
\[
\mu = \mu[K_1,\cdot] \Psi(K_1) + \mu[K_2,\cdot] \Psi(K_2)
\]
for all unimodular measures $\mu$ sustained by $Y$. That is, $\mu$ is a convex combination of the measures $\Psi(K_1)$ and $\Psi(K_2)$. Unfortunately, this is not as simple for more general graphs, but the unimodularity of $\mu$ is enough to prove it.

%
%
\subsection{The union of multiple copies of a connected judicial graph sustains a unique unimodular measure}

There is a special case though: a union of multiple copies of a connected judicial graph. As the following proposition states, such a graph sustains a unique unimodular measure, even though it is not connected.

\begin{prop}\label{multiple_copies_of_connected_graph_is_judicial}
If $X$ is a connected judicial graph, then for all positive integers $b$, the graph $bX$ is judicial, and $\Psi(X)$ is the unique unimodular measure sustained by $bX$.
\end{prop}

\begin{proof}
Suppose that $X$ is a connected judicial graph, let $\mu = \Psi(X)$, and let $b$ be a positive integer. Since $[X,x] = [(bX)_x,x]$ where $(bX)_x$ is the connected component of $bX$ containing $x$, it follows that the set of rooted connected components of $X$ coincides with that of $bX$. Clearly, $\mu$ is sustained by $bX$. To prove uniqueness, assume that $\nu \in \U$ is sustained by $bX$ too. Then $\nu$ is sustained by $X$, which means $\nu = \mu$ by Theorem \ref{uniqueness_of_sustained_unimodular_measures}. Hence $\mu$ is the unique unimodular measure sustained by $bX$.
\end{proof}

%
%
\subsection{Graphs with judicial components are judicial}

In general, a disconnected judicial graph does not sustain a unique unimodular measure. Instead, as the following result states, every unimodular measure sustained by such a graph is a convex combination.

%
%
\begin{framed}
\begin{theo}\label{judiciality_of_disconnected_graphs}
Let $I$ be a subset of $\NN$. Suppose that $X = \sum_{k \in I} b_kX^k$ where $\{X^k ~:~ k \in I\}$ is a set of pairwise distinct connected judicial graphs and $\{b_k ~:~ k \in I\}$ is a set of positive integers. Let $\mu^k = \Psi(X^k)$ for all $k \in I$. If $\mu \in \U$ is sustained by $X$, then $\mu$ is the convex combination
\[
\mu = \sum_{k \in I} \left(\sum_{x \in \Rcc(k)} \mu[X^k,x]\right) \mu^k
\]
where $x \in \Rcc(k)$ is shorthand for $[X^k,x] \in \Rcc(X^k)$.
\end{theo}
\end{framed}

\begin{figure}[ht]
\begin{center}
\includegraphics[scale=1.2]{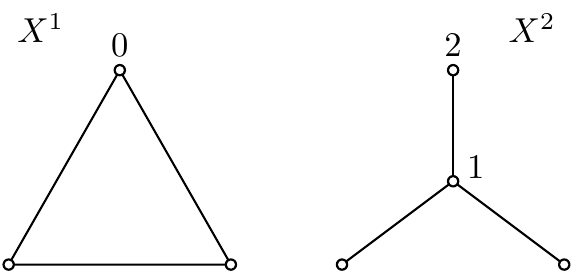}
\caption[The graph $X$ whose connected components are $X^1$ and $X^2$]{The graph $X$ whose connected components are $X^1$ and $X^2$.}
\label{z3_and_t1}
\end{center}
\end{figure}

To gain some intuition for this result, consider the following example. Denote by $X^1$ and $X^2$ the graphs in Figure \ref{z3_and_t1}, let $\mu^1 = \Psi(X^1)$ and $\mu^2 = \Psi(X^2)$, and let $X = X^1 + X^2$. Here $\Rcc(X^1) = \{[X^1,0]\}$ and $\Rcc(X^2) = \{[X^2,1],[X^2,2]\}$. Note that $\mu^1[X^1,0] = 1$, $\mu^2[X^2,1] = 1/4$, and $\mu^2[X^2,2] = 3/4$. As in the hypothesis of Theorem \ref{judiciality_of_disconnected_graphs}, let $\mu$ be a unimodular measure sustained by $X$. In order to use unimodularity in this case, it suffices to define a pair of nonnegative functions $g$ and $h$ on $\dGr$ as follows: $g[X^2,1,j] = 1$ and
\[
h[X^2,1,j] = \frac{\mu[X^2,1]}{\mu^2[X^2,1]}
\]
for all $j \in N(1)$, and $g$ and $h$ are zero otherwise. Since $\mu$ is unimodular and $N(2) = \{1\}$, it follows that
\[
g[X^2,1,2] \cdot \mu[X^2,2] + \sum_{j \in N(1)} g[X^2,j,1] \cdot \mu[X^2,1] = \mu[X^2,2]
\]
and
\[
g[X^2,2,1] \cdot \mu[X^2,2] + \sum_{j \in N(1)} g[X^2,1,j] \cdot \mu[X^2,1] = |N(1)|\mu[X^2,1]
\]
are equal. That is, $\mu[X^2,2] = 3\mu[X^2,1]$. Similarly, the unimodularity of $\mu^2$ implies that
\[
h[X^2,2,1] \cdot \mu^2[X^2,2] + \sum_{j \in N(1)} h[X^2,1,j] \cdot \mu^2[X^2,1] = |N(1)|\left(\frac{\mu[X^2,1]\mu^2[X^2,1]}{\mu^2[X^2,1]}\right)
\]
and
\[
h[X^2,1,2] \cdot \mu^2[X^2,2] + \sum_{j \in N(1)} h[X^2,j,1] \cdot \mu^2[X^2,1] = \frac{\mu[X^2,1]\mu^2[X^2,2]}{\mu^2[X^2,1]}
\]
are equal, which means $3\mu[X^2,1]\mu^2[X^2,1] = \mu[X^2,1]\mu^2[X^2,2]$. Thus
\[
\mu[X^2,1]\mu^2[X^2,2] = \mu[X^2,2]\mu^2[X^2,1].
\]
It remains to show that
\[
\mu = \mu[X^1,0]\mu^1 + (\mu[X^2,1] + \mu[X^2,2]) \cdot \mu^2,
\]
which may be accomplished by evaluating the right-hand side at $[X^1,0]$, $[X^2,1]$, and $[X^2,2]$. Doing so at $[X^1,0]$ is trivial because $\mu^2[X^1,0] = 0$ and $\mu^1[X^1,0] = 1$. However, let us use the equality obtained above, and apply the right-hand side to $[X^2,1]$:
\begin{align*}
\mu[X^1,0]\mu^1[X^2,1] &+ (\mu[X^2,1] + \mu[X^2,2]) \cdot \mu^2[X^2,1]\\
                       &= \mu[X^2,1]\mu^2[X^2,1] + \mu[X^2,2]\mu^2[X^2,1]\\
                       &= \mu[X^2,1]\mu^2[X^2,1] + \mu[X^2,1]\mu^2[X^2,2]\\
                       &= \mu[X^2,1] \cdot (\mu^2[X^2,1] + \mu^2[X^2,2])\\
                       &= \mu[X^2,1] \left(\frac{1}{4} + \frac{3}{4}\right)\\
                       &= \mu[X^2,1]
\end{align*}
because $\mu^1[X^2,1] = 0$. An identical sequence of equalities proves that
\[
\mu[X^1,0]\mu^1[X^2,2] + (\mu[X^2,1] + \mu[X^2,2]) \cdot \mu^2[X^2,2] = \mu[X^2,2].
\]
Hence the conclusion of Theorem \ref{judiciality_of_disconnected_graphs} holds for this example. As the reader can guess, this idea of exchanging the arguments of $\mu$ and $\mu^k$ is the basis of the proof of Theorem \ref{judiciality_of_disconnected_graphs}. The ability to perform such an exchange is entirely obtained using the definition of unimodularity by specifying functions similar to $g$ and $h$ in the example above. The remainder of the section involves fleshing out these technical results, and eventually demonstrating Theorem \ref{judiciality_of_disconnected_graphs} followed by a few interesting corollaries.

\begin{lem}\label{lemma_lemma_judiciality_of_disconnected_graphs}
Under the assumptions of Theorem \ref{judiciality_of_disconnected_graphs}, $\mu$ is unimodular if and only if
\[
\sum_{k \in I} \sum_{i \in \Rcc(k)} \sum_{j \in N(i)} f[X^k,i,j] \cdot \mu[X^k,i] = \sum_{k \in I} \sum_{i \in \Rcc(k)} \sum_{j \in N(i)} f[X^k,j,i] \cdot \mu[X^k,i]
\]
for all nonnegative functions $f$ on $\dGr$; and $\mu^l$ is unimodular if and only if
\[
\sum_{i \in \Rcc(l)} \sum_{j \in N(i)} f[X^l,i,j] \cdot \mu^l[X^l,i] = \sum_{i \in \Rcc(l)} \sum_{j \in N(i)} f[X^l,j,i] \cdot \mu^l[X^l,i]
\]
for all nonnegative functions $f$ on $\dGr$ and $l \in I$.
\end{lem}

\begin{proof}
Since $\{X^k ~:~ k \in I\}$ consists of pairwise distinct connected judicial graphs, it follows that
\[
\Rcc(X) = \bigcup_{k \in I} \Rcc(X^k)
\]
is a disjoint union. Then
\[
\sum_{[X_x,x] \in \Rcc(X)} \sum_{y \in N(x)} f[X_x,x,y] \cdot \mu[X_x,x] = \sum_{k \in I} \sum_{i \in \Rcc(k)} \sum_{j \in N(i)} f[X^k,i,j] \cdot \mu[X^k,i]
\]
and
\[
\sum_{[X_x,x] \in \Rcc(X)} \sum_{y \in N(x)} f[X_x,y,x] \cdot \mu[X_x,x] = \sum_{k \in I} \sum_{i \in \Rcc(k)} \sum_{j \in N(i)} f[X^k,j,i] \cdot \mu[X^k,i]
\]
for all nonnegative functions $f$ on $\dGr$. The remainder follows immediately using the definition of unimodularity.
\end{proof}

\begin{lem}\label{lemma_judiciality_of_disconnected_graphs}
Under the assumptions of Lemma \ref{lemma_lemma_judiciality_of_disconnected_graphs},
\[
\mu[X^l,x]\mu^l[X^l,y] = \mu[X^l,y]\mu^l[X^l,x]
\]
for all adjacent vertices $x$ and $y$ such that $[X^l,x]$ and $[X^l,y]$ are rooted connected components of $X^l$, and $l \in I$.
\end{lem}

\begin{proof}
Let $l \in I$, and let $[X^l,x]$ and $[X^l,y]$ be rooted connected components of $X^l$. If $[X^l,x] = [X^l,y]$, the result holds. Suppose that $[X^l,x] \neq [X^l,y]$ and $y \in N(x)$. Denote by $f$ the characteristic function of the set
\[
\{[X^l,w,z] \in \dGr ~:~ w \in [x] ~\text{ and }~ z \in [y]\}
\]
where $[x] = \Aut(X^l)x$ and $[y] = \Aut(X^l)y$ are the orbits of $x$ and $y$, respectively. Define the nonnegative function $f_l$ on $\dGr$ as follows:
\[
f_l[X^l,i,j] = \frac{f[X^l,i,j] \cdot \mu[X^l,i]}{\mu^l[X^l,i]}
\]
for all adjacent vertices $i$ and $j$ of $X^l$, and $f_l = 0$ otherwise. Note that $f[X^l,x,j] = 1$ if and only if $j \in N(x) \cap [y]$, and $f[X^l,j,y] = 1$ if and only if $j \in N(y) \cap [x]$.

Since $\mu$ is unimodular, Lemma \ref{lemma_lemma_judiciality_of_disconnected_graphs} implies that
\begin{align*}
\sum_{k \in I} \sum_{i \in \Rcc(k)} \sum_{j \in N(i)} f[X^k,i,j] \cdot \mu[X^k,i] &= \sum_{i \in \Rcc(l)} \sum_{j \in N(i)} f[X^l,i,j] \cdot \mu[X^l,i]\\
           &= \sum_{j \in N(x)} f[X^l,x,j] \cdot \mu[X^l,x]\\
           &= \sum_{j \in N(x) \cap [y]} \mu[X^l,x]\\
           &= |N(x) \cap [y]|\mu[X^l,x]
\end{align*}
and
\begin{align*}
\sum_{k \in I} \sum_{i \in \Rcc(k)} \sum_{j \in N(i)} f[X^k,j,i] \cdot \mu[X^k,i] &= \sum_{i \in \Rcc(l)} \sum_{j \in N(i)} f[X^l,j,i] \cdot \mu[X^l,i]\\
           &= \sum_{j \in N(y)} f[X^l,j,y] \cdot \mu[X^l,y]\\
           &= \sum_{j \in N(y) \cap [x]} \mu[X^l,y]\\
           &= |N(y) \cap [x]|\mu[X^l,y]
\end{align*}
are equal, and so
\[
|N(x) \cap [y]|\mu[X^l,x] = |N(y) \cap [x]|\mu[X^l,y].
\]
Similarly, the unimodularity of $\mu^l$ tells us that
\begin{align*}
\sum_{i \in \Rcc(l)} \sum_{j \in N(i)} f_l[X^l,i,j] \cdot \mu^l[X^l,i] &= \sum_{j \in N(x)} f_l[X^l,x,j] \cdot \mu^l[X^l,x]\\
&= \sum_{j \in N(x) \cap [y]} \left(\frac{\mu[X^l,x]}{\mu^l[X^l,x]}\right) \mu^l[X^l,x]\\
&= \sum_{j \in N(x) \cap [y]} \mu[X^l,x]\\
&= |N(x) \cap [y]|\mu[X^l,x]
\end{align*}
and
\begin{align*}
\sum_{i \in \Rcc(l)} \sum_{j \in N(i)} f_l[X^l,j,i] \cdot \mu^l[X^l,i] &= \sum_{j \in N(y)} f_l[X^l,j,y]) \cdot \mu^l[X^l,y]\\
&= \sum_{j \in N(y) \cap [x]} \left(\frac{\mu[X^l,j]}{\mu^l[X^l,j]}\right) \mu^l[X^l,y]\\
&= \sum_{j \in N(y) \cap [x]} \left(\frac{\mu[X^l,x]}{\mu^l[X^l,x]}\right) \mu^l[X^l,y]\\
&= |N(y) \cap [x]|\left(\frac{\mu[X^l,x]\mu^l[X^l,y]}{\mu^l[X^l,x]}\right)
\end{align*}
are equal where $\mu[X^l,j] = \mu[X^l,x]$ and $\mu^l[X^l,j] = \mu^l[X^l,x]$ because $j \in [x]$. Thus
\[
|N(x) \cap [y]|\mu[X^l,x] = |N(y) \cap [x]|\left(\frac{\mu[X^l,x]\mu^l[X^l,y]}{\mu^l[X^l,x]}\right).
\]
Observe that $N(y) \cap [x]$ is nonempty because $x \in N(y)$. By combining the equations above,
\[
|N(y) \cap [x]|\mu[X^l,y] = |N(y) \cap [x]|\left(\frac{\mu[X^l,x]\mu^l[X^l,y]}{\mu^l[X^l,x]}\right),
\]
and so $\mu[X^l,y]\mu^l[X^l,x] = \mu[X^l,x]\mu^l[X^l,y]$ because $|N(y) \cap [x]|$ is nonzero.
\end{proof}

\begin{cor}\label{corollary_judiciality_of_disconnected_graphs}
Under the assumptions of Lemma \ref{lemma_judiciality_of_disconnected_graphs},
\[
\mu[X^l,x]\mu^l[X^l,y] = \mu[X^l,y]\mu^l[X^l,x]
\]
for all vertices $x$ and $y$ of $X^l$, and $l \in I$.
\end{cor}

%
%
\begin{proof}
Let $l \in I$, and let $x$ and $y$ be vertices of $X^l$. Since $X^l$ is connected, there is a path $(x_0,x_1,\ldots,x_m)$ with $x_0 = x$ and $x_m = y$ in $X^l$. Note that $[X^l,x_i] \in \Rcc(X^l)$ for all $i \in \{0,1,\ldots,m\}$. By Lemma \ref{lemma_judiciality_of_disconnected_graphs},
\begin{empheq}[left=\empheqlbrace]{align*}
\mu[X^l,x_0]\mu^l[X^l,x_1]     &= \mu[X^l,x_1]\mu^l[X^l,x_0]\\
\mu[X^l,x_1]\mu^l[X^l,x_2]     &= \mu[X^l,x_2]\mu^l[X^l,x_1]\\
                               &\setbox0\hbox{=}\mathrel{\makebox[\wd0]{\vdots}}\\
\mu[X^l,x_{m-1}]\mu^l[X^l,x_m] &= \mu[X^l,x_m]\mu^l[X^l,x_{m-1}],
\end{empheq}
and so
\[
\frac{\mu[X^l,x_0]}{\mu[X^l,x_1]} \frac{\mu[X^l,x_1]}{\mu[X^l,x_2]} \cdots \frac{\mu[X^l,x_{m-1}]}{\mu[X^l,x_m]} = \frac{\mu^l[X^l,x_0]}{\mu^l[X^l,x_1]} \frac{\mu^l[X^l,x_1]}{\mu^l[X^l,x_2]} \cdots \frac{\mu^l[X^l,x_{m-1}]}{\mu^l[X^l,x_m]}.
\]
Hence $\mu[X^l,x]\mu^l[X^l,y] = \mu[X^l,y]\mu^l[X^l,x]$, as required.
\end{proof}

\begin{proof}[Proof of Theorem \ref{judiciality_of_disconnected_graphs}]
Suppose that $\mu \in \U$ is sustained by $X$. To show that
\[
\mu = \sum_{k \in I} \left(\sum_{x \in \Rcc(k)} \mu[X^k,x]\right) \mu^k,
\]
it suffices to prove that the left-hand side and the right-hand side agree on the rooted connected components of $X$. Let $l \in I$, and let $[X^l,y] \in \Rcc(X^l)$. Then
\begin{align*}
\sum_{k \in I} \left(\sum_{x \in \Rcc(k)} \mu[X^k,x]\right) \mu^k[X^l,y] &= \sum_{x \in \Rcc(l)} \mu[X^l,x]\mu^l[X^l,y]\\
         &= \sum_{x \in \Rcc(l)} \mu[X^l,y]\mu^l[X^l,x]\\
         &= \sum_{x \in \Rcc(l)} \mu^l[X^l,x] \cdot \mu[X^l,y]\\
         &= 1 \cdot \mu[X^l,y]\\
         &= \mu[X^l,y]
\end{align*}
where the second equality holds by Corollary \ref{corollary_judiciality_of_disconnected_graphs}.
\end{proof}

%
%
\subsection{The case of disconnected finite graphs}

To put the main result into perspective, consider the special case of finite graphs and their laws. The particular case of the following result is demonstrated in the author's Honours project \cite{artemenko11a}. Let us instead show how it arises from Theorem \ref{judiciality_of_disconnected_graphs}. Recall that $\Psi(X)$ is defined by $\Psi(X)[X_x,x] = |\Aut(X)x|/|V(X)|$ for all $[X_x,x] \in \Rcc(X)$ if $X$ is a finite graph.

\begin{cor}\label{judiciality_of_disconnected_graphs_corollary}
Suppose that $X = \sum_{k=1}^\omega b_kX^k$ where $\{X^k ~:~ 1 \leq k \leq \omega\}$ is a set of pairwise distinct finite connected graphs and $\{b_k ~:~ 1 \leq k \leq \omega\}$ is a set of positive integers for some positive integer $\omega$. Then
\[
\Psi(X) = \sum_{k=1}^\omega \left(\frac{b_k|V(X^k)|}{|V(X)|}\right) \Psi(X^k).
\]
In particular, if $m$ and $n$ are nonnegative integers with $(m,n) \neq (0,0)$, and $G$ and $H$ are distinct finite connected graphs, then
\[
\Psi(mG + nH) = \frac{m|V(G)| \cdot \Psi(G) + n|V(H)| \cdot \Psi(H)}{m|V(G)| + n|V(H)|}.
\]
\end{cor}

To ease the proof of Corollary \ref{judiciality_of_disconnected_graphs_corollary}, it is wise to isolate the following technical result, which relates the size of an orbit of $X$ to that of an orbit of one of its components.

\begin{lem}\label{lemma_judiciality_of_disconnected_graphs_corollary}
Let $X$ be a finite graph. Suppose that $X = mG + Y$ for some positive integer $m$ where $G$ is a connected graph, which is not a connected component of $Y$. Then $|\Aut(X)o| = m|\Aut(G)o|$ for all $o \in V(G)$.
\end{lem}

\begin{proof}
Let $o \in V(G)$. Since $X = mG + Y$, it is possible to write
\[
V(X) = (V(G) \times \{0,1,\ldots,m - 1\}) \cup V(Y),
\]
and identify the vertices $(o,0)$ and $o$. It suffices to prove that $\Aut(X)(o,0) = \Aut(G)o \times \{0,1,\ldots,m - 1\}$. Given $(x,i) \in \Aut(X)(o,0)$, there is an automorphism $\varphi$ on $X$ such that $\varphi(o,0) = (x,i)$, and so
\[
[G \times 0,o] = [X_o,o] = [\varphi(X_o),\varphi(o)] = [X_{\varphi(o)},\varphi(o)] = [X_{(x,i)},(x,i)]
\]
where $G \times 0$ is the $0$th copy of $G$ in $X$. Since $G$ is not a connected component of $Y$, their sets of rooted connected components are disjoint, which means $X_{(x,i)} = G \times i$, and so $x \in V(G)$. Furthermore, the inclusion $\iota : G \to G \times 0$, the projection $\pi : G \times i \to G$, and the restriction $\varphi_G$ of $\varphi$ to $G \times 0$ whose image is
\[
\varphi_G(G \times 0) = \varphi(X_o) = X_{\varphi(o)} = G \times i
\]
are isomorphisms. Their composition $\pi \circ \varphi_G \circ \iota$ is an automorphism on $G$ with
\[
(\pi \circ \varphi_G \circ \iota)(o) = \pi(\varphi_G(o,0)) = \pi(\varphi(o,0)) = \pi(x,i) = x,
\]
and so $x \in \Aut(G)o$. On the other hand, assume that $i \in \{0,1,\ldots,m - 1\}$ and $x \in \Aut(G)o$. Then there is an automorphism $\varphi$ on $G$ such that $\varphi(o) = x$. Define the function $\hat\varphi : V(X) \to V(X)$ as follows: $\hat\varphi(y,j) = (\varphi(y),j)$ for all $(y,j) \in V(G) \times \{0,1,\ldots,m - 1\}$, and $\hat\varphi$ is the identity on $V(Y)$. Clearly, $\hat\varphi$ is an automorphism on $X$, and so is the function $s$, which exchanges the copies $G \times 0$ and $G \times i$ leaving the rest fixed. Thus the composition $s \circ \hat\varphi$ is an automorphism on $X$ such that $(s \circ \hat\varphi)(o,0) = s(\varphi(o),0) = s(x,0) = (x,i)$. Hence $(x,i) \in \Aut(X)(o,0)$, and so
\[
|\Aut(X)o| = |\Aut(X)(o,0)| = |\Aut(G)o \times \{0,1,\ldots,m - 1\}| = m|\Aut(G)o|.\qedhere
\]
\end{proof}

\begin{proof}[Proof of Corollary \ref{judiciality_of_disconnected_graphs_corollary}]
Let $\mu = \Psi(X)$ for convenience. Certainly, $\mu$ is unimodular and sustained by $X$. By Theorem \ref{judiciality_of_disconnected_graphs},
\[
\mu = \sum_{k=1}^\omega \left(\sum_{x \in \Rcc(k)} \mu[X^k,x]\right) \Psi(X^k)
\]
where $x \in \Rcc(k)$ is shorthand for $[X^k,x] \in \Rcc(X^k)$. Let $k \in \{1,2,\ldots,\omega\}$. By Lemma \ref{lemma_judiciality_of_disconnected_graphs_corollary}, $|\Aut(X)y| = b_k|\Aut(X^k)y|$ for all $y \in V(X^k)$. Then
\[
b_k|V(X^k)| = \sum_{x \in V(X^k)} b_k = \sum_{x \in \Rcc(k)} b_k|\Aut(X^k)x| = \sum_{x \in \Rcc(k)} |\Aut(X)x|,
\]
and so
\[
\sum_{x \in \Rcc(k)} \mu[X^k,x] = \sum_{x \in \Rcc(k)} \Psi(X)[X^k,x] = \sum_{x \in \Rcc(k)} \frac{|\Aut(X)x|}{|V(X)|} = \frac{b_k|V(X^k)|}{|V(X)|}
\]
using the definition of the law of $X$, which demonstrates the result.
\end{proof}

%
%
\subsection{Non-null components of a judicial graph are judicial}

Having shown that the union of judicial graphs forms another such graph, the reader may wonder if every disconnected judicial graph has judicial components. Unfortunately, by assigning a measure of zero to the lawless components, it is clear that this is not true. However, with an additional assumption on the connected components, a weaker statement is true.

\begin{prop}
Let $X$ be a judicial graph that sustains a measure $\mu \in \U$. If $Y$ is a connected component of $X$ and $\mu[Y,y] \neq 0$ for some $y \in V(Y)$, then $Y$ is judicial.
\end{prop}

\begin{proof}
Let $X$ be a judicial graph, and let $Y$ be a connected component of $X$. Suppose that $\mu \in \U$ is sustained by $X$. Clearly,
\[
a = \sum_{[Y,x] \in \Rcc(Y)} \mu[Y,x]
\]
is nonzero because $\mu[Y,y] \neq 0$ for some $y \in V(Y)$. Denote by $\chi$ the characteristic function of $\Rcc(Y)$. Note that $\nu = (\mu \cdot \chi)/a$ is a measure sustained by $Y$. It remains to show that $\nu$ is unimodular. Let $f$ be a nonnegative function on $\dGr$ with $f(\dGr \setminus \BRcc(Y)) = \{0\}$. Define the function $g : \dGr \to \RR$ by
\[
g[G,u,v] = \frac{f[G,u,v] \cdot \chi[G,u]}{a}
\]
for all $[G,u,v] \in \dGr$. Clearly, $g(\dGr \setminus \BRcc(X)) = \{0\}$ because $\BRcc(Y) \subseteq \BRcc(X)$. Then
\begin{align*}
\sum_{[X_x,x] \in \Rcc(X)} \sum_{y \in N(x)} g[X_x,x,y] \cdot \mu[X_x,x] &= \sum_{[Y,x] \in \Rcc(Y)} \sum_{y \in N(x)} g[Y,x,y] \cdot \mu[Y,x]\\
&= \sum_{[Y,x] \in \Rcc(Y)} \sum_{y \in N(x)} \left(\frac{f[Y,x,y] \cdot \chi[Y,x]}{a}\right) \mu[Y,x]\\
&= \sum_{[Y,x] \in \Rcc(Y)} \sum_{y \in N(x)} f[Y,x,y] \cdot \nu[Y,x]
\end{align*}
and
\begin{align*}
\sum_{[X_x,x] \in \Rcc(X)} \sum_{y \in N(x)} g[X_x,y,x] \cdot \mu[X_x,x] &= \sum_{[Y,x] \in \Rcc(Y)} \sum_{y \in N(x)} g[Y,y,x] \cdot \mu[Y,x]\\
&= \sum_{[Y,x] \in \Rcc(Y)} \sum_{y \in N(x)} \left(\frac{f[Y,y,x] \cdot \chi[Y,y]}{a}\right) \mu[Y,x]\\
&= \sum_{[Y,x] \in \Rcc(Y)} \sum_{y \in N(x)} \left(\frac{f[Y,y,x] \cdot \chi[Y,x]}{a}\right) \mu[Y,x]\\
&= \sum_{[Y,x] \in \Rcc(Y)} \sum_{y \in N(x)} f[Y,y,x] \cdot \nu[Y,x]
\end{align*}
are equal because $\mu$ is unimodular. Hence $\nu$ is unimodular too.
\end{proof}

%% file: weak_limits_are_invariant_under_negligence.tex
%
%
\section{Weak limits are invariant under negligence}

Although this section bears a silly name, the concept is meaningful. We show that the weak limit of a sequence of finite graphs remains the same if a so-called negligible subgraph of each term is deleted. Using this result, we deduce several simple consequences.

\begin{figure}[ht]
\begin{center}
\includegraphics[scale=1.2]{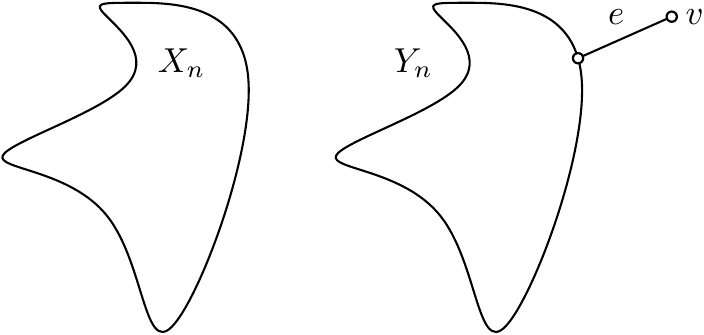}
\caption[The graph $Y_n$ is $X_n$ with an additional vertex $v$ and an edge $e$]{The graph $Y_n$ is $X_n$ with an additional vertex $v$ and an edge $e$.}
\label{graph_and_graph_plus_edge}
\end{center}
\end{figure}

Let us begin with an example. Consider the graphs in Figure \ref{graph_and_graph_plus_edge}, which are terms in the sequences $(X_n)$ and $(Y_n)$ of finite connected graphs where $Y_n$ is $X_n$ with an additional edge $e$ and vertex $v$ for all positive integers $n$. In general, the orbits of these graphs are vastly different. This means that it is not wise to compare the laws of $X_n$ and $Y_n$ by definition. Instead, let $f \in \C(\Gr)$, and recall that
\[
\int f ~d\Psi(X_n) = \sum_{x \in V(X_n)} \frac{f[X_n,x]}{|V(X_n)|}
\]
where we do not distinguish orbits, and simply sum over the vertices of $X_n$. Furthermore, the distance between $[X_n,x]$ and $[Y_n,x]$ is small if $x \in V(X_n)$ is far from $v$, that is, if $d_{Y_n}(x,v) > r$ for some nonnegative integer $r$. Using the continuity of $f$, it is possible to deduce that the distance between $\int f ~d\Psi(X_n)$ and $\int f ~d\Psi(Y_n)$ tends to zero as $n$ grows. In this case, the sequences $(\Psi(X_n))$ and $(\Psi(Y_n))$ have the same weak limit if it exists.

%
%
\subsection{Example: the graph obtained by attaching trees to cycles}

\begin{figure}[ht]
\begin{center}
\includegraphics[scale=1.2]{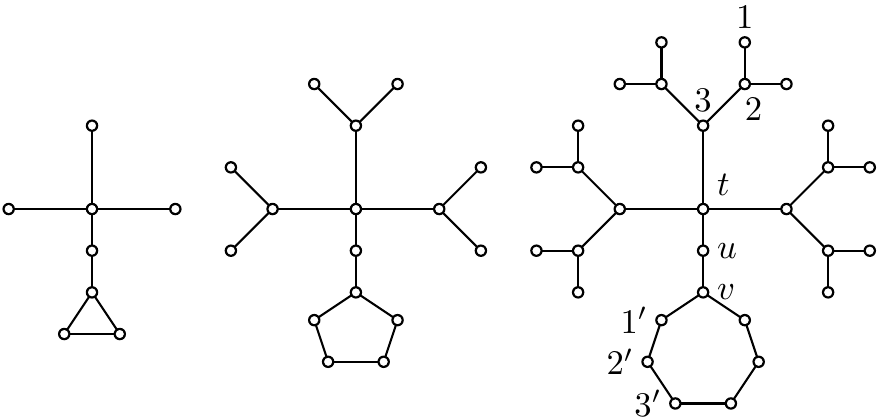}
\caption[The terms $X_1$, $X_2$, and $X_3$ of the sequence $(X_n)$]{The terms $X_1$, $X_2$, and $X_3$ of the sequence $(X_n)$.}
\label{tree_plus_cycle}
\end{center}
\end{figure}

To better understand this situation, consider the following specific example. Recall that $T$ is the $3$-regular tree, and $t$ is a vertex of $T$. Denote by $(X_n)$ the sequence illustrated in Figure \ref{tree_plus_cycle} and defined as follows: $X_n$ is obtained by ``attaching'' the root of $T_n = B_T(t,n)$ to a vertex of the cycle $\ZZ_{2n+1}$ using a path on three vertices for all positive integers $n$. The set of rooted connected components is the union
\[
\{[X_n,t],[X_n,u],[X_n,v]\} \cup \{[X_n,i] ~:~ 1 \leq i \leq n\} \cup \{[X_n,i'] ~:~ 1 \leq i \leq n\},
\]
and
\[
|V(X_n)| = |V(T_n)| + |V(\ZZ_{2n+1})| + 1 = (3 \cdot 2^n - 2) + (2n + 1) + 1 = 3 \cdot 2^n + 2n
\]
for all positive integers $n$. Observe that
\[
\Psi(X_n)[X_n,i] = \frac{|\Aut(X_n)i|}{|V(X_n)|} = \frac{3 \cdot 2^{n-i}}{3 \cdot 2^n + 2n},
\]
\[
\Psi(X_n)[X_n,i'] = \frac{|\Aut(X_n)i'|}{|V(X_n)|} = \frac{2}{3 \cdot 2^n + 2n},
\]
and
\[
\Psi(X_n)[X_n,t] = \Psi(X_n)[X_n,u] = \Psi(X_n)[X_n,v] = \frac{1}{3 \cdot 2^n + 2n}
\]
for all positive integers $n$ and $i \in \{1,2,\ldots,n\}$. Individually, the weak limits of $(\Psi(T_n))$ and $(\Psi(\ZZ_{2n+1}))$ are $\mu_S$ and the Dirac measure $\delta_\ZZ$, respectively. It is reasonable to assume that the weak limit of $(\Psi(X_n))$ is sustained by the graphs $S$ and $\ZZ$, and related to $\mu_S$ and $\delta_\ZZ$. In fact, it is a convex combination of these two measures. However, the goal is to determine the appropriate coefficients of this combination.

Given a function $f \in \C(\Gr)$,
\[
\int f ~d\Psi(X_n) = \frac{f[X_n,t] + f[X_n,u] + f[X_n,v]}{3 \cdot 2^n + 2n} + \sum_{i=1}^n \frac{(3 \cdot 2^{n-i})f[X_n,i] + 2f[X_n,i']}{3 \cdot 2^n + 2n}
\]
for all positive integers $n$. Note that
\[
\lim_{n \to \infty} \int f ~d\Psi(X_n) = a \lim_{n \to \infty} \left(\sum_{i=1}^n \frac{f[X_n,i]}{2^i}\right) + b \lim_{n \to \infty} \left(\sum_{i=1}^n \frac{2f[X_n,i']}{2n}\right)
\]
where
\[
a = \lim_{n \to \infty} \left(\frac{3 \cdot 2^n}{3 \cdot 2^n + 2n}\right) = 1
\]
and
\[
b = \lim_{n \to \infty} \left(\frac{2n}{3 \cdot 2^n + 2n}\right) = 0.
\]
Since $X_n$ is similar to $T_n$ and $P_n$ in an appropriate region outside of a neighbourhood of $t$, let us assume that $[X_n,i]$ and $[T_n,i]$ are close, and $[X_n,i']$ and $[P_n,i']$ are close in the limit. Furthermore,
\[
\int f ~d\mu_S = \lim_{n \to \infty} \int f ~d\Psi(T_n) = \lim_{n \to \infty} \left(\sum_{i=1}^n \frac{f[T_n,i]}{2^i}\right)
\]
and
\[
\int f ~d\delta_\ZZ = \lim_{n \to \infty} \int f ~d\Psi(P_n) = \lim_{n \to \infty} \left(\sum_{i=1}^n \frac{2f[P_n,i']}{2n}\right),
\]
which means
\[
\lim_{n \to \infty} \int f ~d\Psi(X_n) = a \int f ~d\mu_S + b \int f ~d\delta_\ZZ
\]
where $a = 1$ and $b = 0$, as previously calculated. Thus the weak limit of $(\Psi(X_n))$ is simply $\mu_S$. Loosely speaking, this happens because the cycles do not grow as quickly as the trees, and so the measure of the cycles becomes negligible. However, this is no longer true if the cycles grow exponentially to keep up with the size of the trees. The weak limit of such a sequence is strictly sustained by both $S$ and $\ZZ$.

%
%
\subsection{Removal of a negligible sequence preserves the weak limit}

\begin{figure}[ht]
\begin{center}
\includegraphics[scale=1.2]{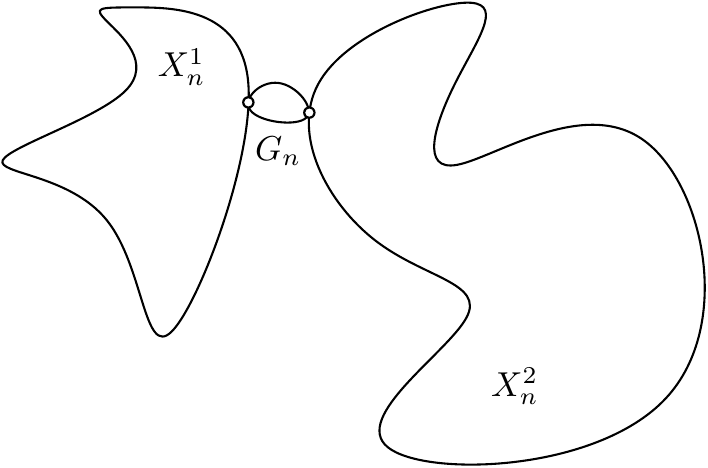}
\caption[The graph $X_n$ with $X_n \setminus G_n = X_n^1 + X_n^2$]{The graph $X_n$. The connected components of $X_n \setminus G_n$ are $X_n^1$ and $X_n^2$.}
\label{graph_connected_to_graph_by_small_graph}
\end{center}
\end{figure}

As shown in Figure \ref{graph_connected_to_graph_by_small_graph}, it is possible to consider more complicated examples of this phenomenon. With an additional stipulation, similar reasoning may be used to prove that the weak limits of $(\Psi(X_n))$ and $(\Psi(X_n^1 + X_n^2))$ are equal if they exist. The requirement is that the ratio $|V(G_n)|/|V(X_n)|$ must vanish as $n$ tends to infinity, which intuitively means the ``link'' $G_n$ eventually disappears, and $X_n$ simply becomes $X_n^1 + X_n^2$. In fact, an even more general statement may be made.

\begin{defn}
Let $(X_n)$ be a sequence of finite graphs. A sequence of graphs $(G_n)$ is \emph{negligible in $(X_n)$}\index{negligible} if $G_n$ is a subgraph of $X_n$ for all positive integers $n$, and
\[
\lim_{n \to \infty} \frac{|V(G_n)|}{|V(X_n)|} = 0.
\]
\end{defn}

%
%
\begin{framed}
\begin{theo}\label{weak_limits_are_invariant_under_negligence}
Suppose that $(X_n)$ and $(G_n)$ are sequences of finite graphs such that $(G_n)$ is negligible in $(X_n)$. Then
\[
\lim_{n \to \infty} \left|\int f ~d\Psi(X_n) - \int f ~d\Psi(X_n \setminus G_n)\right| = 0
\]
for all $f \in \C(\Gr)$. Furthermore, $(\Psi(X_n))$ converges weakly to $\mu$ if and only if $(\Psi(X_n \setminus G_n))$ does too.
\end{theo}
\end{framed}

\begin{figure}[ht]
\begin{center}
\includegraphics[scale=1.2]{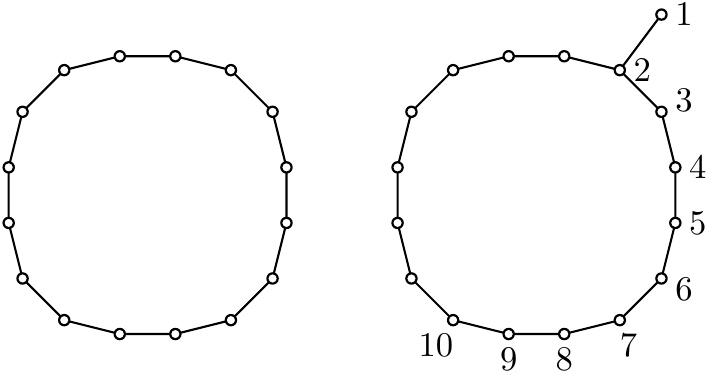}
\caption[A cycle $\ZZ_{16}$, and a copy of it with an additional vertex and edge]{A cycle $\ZZ_{16}$, and a copy of it with an additional vertex and edge.}
\label{cycle_and_cycle_plus_edge}
\end{center}
\end{figure}

Unfortunately, it is not readily possible to use the definition of the law of a finite graph to prove this result. As mentioned before, and as Figure \ref{cycle_and_cycle_plus_edge} demonstrates, two graphs that appear similar may have very different orbits. Indeed, the cycle $\ZZ_{16}$ is vertex-transitive, but the graph on the right has ten distinct orbits, represented by the first ten positive integers. However, it is possible to ignore orbits and focus on the similar portions of the graphs, a procedure which is made precise using neighbourhoods.

\begin{defn}
Let $X$ be a graph, let $G$ be a subgraph of $X$, and let $r$ be a nonnegative integer. Denote by $\N_X(G,r)$\index{NXGr@$\N_X(G,r)$} the \emph{$r$-neighbourhood of $G$ in $X$}\index{rneighbourhood@$r$-neighbourhood}. That is,
\[
\N_X(G,r) = \{x \in V(X) ~:~ d(x,G) \leq r\}
\]
where $d(x,G) = \min\{d(x,y) ~:~ y \in V(G)\}$.
\end{defn}

\begin{prop}\label{the_nbhd_is_bounded}
Recall that $\Delta$ is a positive integer such that $\deg_Y(y) \leq \Delta$ for all $[Y,y] \in \Gr$. Let $G$ be a finite subgraph of a graph $X$. If $r$ is a nonnegative integer, then $|\N_X(G,r)| \leq (\Delta + 1)^r \cdot |V(G)|$.
\end{prop}

\begin{proof}
Suppose that $r$ is a nonnegative integer and $x \in V(X)$. Let
\[
S_X(x,i) = B_X(x,i) \setminus B_X(x,i - 1)
\]
for all $i \in \{1,2,\ldots,r\}$, and observe that
\[
|B_X(x,r)| = 1 + \sum_{i=1}^r |S_X(x,i)| \leq 1 + \sum_{i=1}^r \Delta(\Delta - 1)^{i-1} = 1 + \frac{\Delta ((\Delta - 1)^r - 1)}{\Delta - 2}
\]
where the right-hand side is at most $(\Delta + 1)^r$ by induction on $r$. Then
\[
|\N_X(G,r)| = \left|\bigcup_{x \in V(G)} B_X(x,r)\right| \leq \sum_{x \in V(G)} |B_X(x,r)| \leq (\Delta + 1)^r \cdot |V(G)|.\qedhere
\]
\end{proof}

\begin{lem}\label{weak_limits_are_invariant_under_negligence_lemma}
Let $X$ be a graph, and let $G$ be a subgraph of $X$. If $r$ is a nonnegative integer, then
\[
\rho([X_x,x],[(X \setminus G)_x,x]) \leq 2^{-r}
\]
for all $x \in V(X) \setminus \N_X(G,r)$.
\end{lem}

\begin{proof}
Let $r$ be a nonnegative integer, let $x \in V(X) \setminus \N_X(G,r)$, and let $Y = X \setminus G$. Suppose that $B_X(x,r) \cap V(G)$ is nonempty. That is, there is a $z \in B_X(x,r) \cap V(G)$. Since $x \notin \N_X(G,r)$, it follows that $r < d(x,G) \leq d(x,z) \leq r$, which is a contradiction; see Figure \ref{ball_does_not_intersect_subgraph}. Hence $B_X(x,r) \subseteq V(Y)$. Certainly, $B_Y(x,r) \subseteq B_X(x,r)$. If $y \in B_X(x,r)$, there is a path from $x$ to $y$ in $X$ of length at most $r$. Clearly, this path must lie in $B_X(x,r)$, which means it lies in $Y$, and so $y \in B_Y(x,r)$. That is, $B_X(x,r) \subseteq B_Y(x,r)$. It follows that $[B_X(x,r),x] = [B_Y(x,r),x]$ for all $x \in V(X) \setminus \N_X(G,r)$.
\end{proof}

\begin{figure}[ht]
\begin{center}
\includegraphics[scale=1.2]{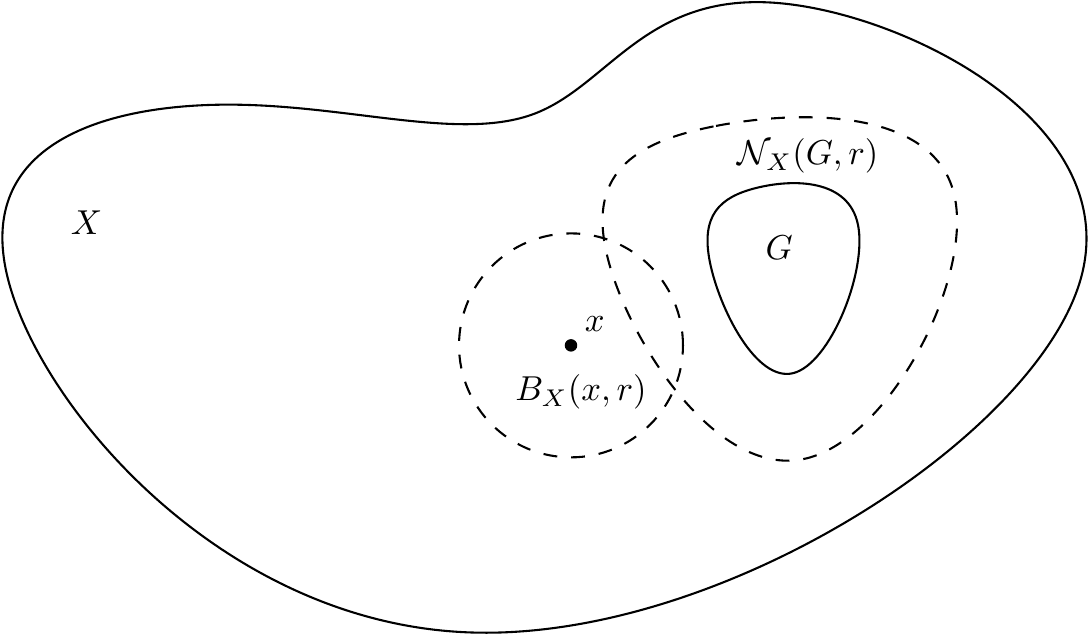}
\caption[If $x$ is not in the $r$-neighbourhood $\N_X(G,r)$, then the ball $B_X(x,r)$ does not intersect the subgraph $G$ of $X$]{If $x$ is not in the $r$-neighbourhood $\N_X(G,r)$, then the ball $B_X(x,r)$ does not intersect the subgraph $G$ of $X$.}
\label{ball_does_not_intersect_subgraph}
\end{center}
\end{figure}

\begin{proof}[Proof of Theorem \ref{weak_limits_are_invariant_under_negligence}]
For convenience, let $Y_n = X_n \setminus G_n$ for all positive integers $n$. Let $\e$ be a positive real number. Since $f \in \C(\Gr)$ where $\Gr$ is compact, it has an upper bound $L$, and it is uniformly continuous:
\[
\exists \delta > 0 ~~ \forall [G,o],[H,p] \in \Gr ~~ \rho([G,o],[H,p]) < \delta ~~\Rightarrow~~ |f[G,o] - f[H,p]| < \e. \tag{$\star$}
\]
Furthermore, $2^{-r} < \delta$ for some positive integer $r$. Let $n$ be a positive integer, and let $\N_n = \N_{X_n}(G_n,r)$ for convenience. By ($\star$) and Lemma \ref{weak_limits_are_invariant_under_negligence_lemma},
\[
|f[(X_n)_x,x] - f[(Y_n)_x,x]| < \e
\]
for all $x \in V(X_n) \setminus \N_n$. Since $V(X_n) = \N_n \cup \N_n^C$ and
\[
V(Y_n) = (V(Y_n) \cap \N_n) \cup (V(Y_n) \cap \N_n^C) = (V(Y_n) \cap \N_n) \cup \N_n^C,
\]
it follows that
\begin{align*}
  &~\left|\sum_{x \in V(X_n)} f[(X_n)_x,x] - \sum_{x \in V(Y_n)} f[(Y_n)_x,x]\right|\\
= &~\left|\sum_{x \in \N_n^C} (f[(X_n)_x,x] - f[(Y_n)_x,x]) + \sum_{x \in \N_n} f[(X_n)_x,x] - \sum_{x \in V(Y_n) \cap \N_n} f[(Y_n)_x,x]\right|\\
\leq &~\sum_{x \in \N_n^C} |f[(X_n)_x,x] - f[(Y_n)_x,x]| + \sum_{x \in \N_n} |f[(X_n)_x,x]| + \sum_{x \in V(Y_n) \cap \N_n} |f[(Y_n)_x,x]|\\
\leq &~\sum_{x \in \N_n^C} |f[(X_n)_x,x] - f[(Y_n)_x,x]| + L \cdot |\N_n| + L \cdot |V(Y_n) \cap \N_n|\\
< &~\e \cdot |V(X_n) \setminus \N_n| + 2L \cdot |\N_n|.
\end{align*}

Then
\begin{align*}
  &~\left|\int f~d\Psi(X_n) - \int f~d\Psi(Y_n)\right|\\
= &~\left|\sum \frac{f[(X_n)_x,x]}{|V(X_n)|} - \sum \frac{f[(Y_n)_x,x]}{|V(Y_n)|}\right|\\
= &~\left|\frac{\sum f[(X_n)_x,x] - \sum f[(Y_n)_x,x]}{|V(X_n)|} + \frac{\sum f[(Y_n)_x,x] \cdot (|V(Y_n)| - |V(X_n)|)}{|V(X_n)| \cdot |V(Y_n)|}\right|\\
\leq &~\frac{|\sum f[(X_n)_x,x] - \sum f[(Y_n)_x,x]|}{|V(X_n)|} + \frac{\sum |f[(Y_n)_x,x]| \cdot ||V(Y_n)| - |V(X_n)||}{|V(X_n)| \cdot |V(Y_n)|}\\
\leq &~\frac{|\sum f[(X_n)_x,x] - \sum f[(Y_n)_x,x]|}{|V(X_n)|} + \frac{L \cdot ||V(Y_n)| - |V(X_n)||}{|V(X_n)|}\\
< &~\frac{\e \cdot |V(X_n) \setminus \N_n|}{|V(X_n)|} + \frac{2L \cdot |\N_n|}{|V(X_n)|} + \frac{L \cdot ||V(Y_n)| - |V(X_n)||}{|V(X_n)|}\\
\leq &~\e + \frac{2L \cdot |\N_n|}{|V(X_n)|} + \frac{L \cdot ||V(Y_n)| - |V(X_n)||}{|V(X_n)|}
\end{align*}
where
\[
\sum \frac{f[(X_n)_x,x]}{|V(X_n)|} = \sum_{x \in V(X_n)} \frac{f[(X_n)_x,x]}{|V(X_n)|}
\]
and
\[
\sum \frac{f[(Y_n)_x,x]}{|V(Y_n)|} = \sum_{x \in V(Y_n)} \frac{f[(Y_n)_x,x]}{|V(Y_n)|}
\]
to simplify notation. It remains to prove that
\[
\lim_{n \to \infty} \frac{|\N_n|}{|V(X_n)|} = 0
\]
and
\[
\lim_{n \to \infty} \frac{||V(Y_n)| - |V(X_n)||}{|V(X_n)|} = 0.
\]
By Proposition \ref{the_nbhd_is_bounded}, $|\N_n| \leq (\Delta + 1)^r \cdot |V(G_n)|$, and so
\[
\lim_{n \to \infty} \frac{|\N_n|}{|V(X_n)|} \leq (\Delta + 1)^r \cdot \lim_{n \to \infty} \frac{|V(G_n)|}{|V(X_n)|} = 0.
\]
On the other hand,
\[
V(Y_n) = V(X_n \setminus G_n) = V(X_n) \setminus V(G_n)
\]
tells us that $||V(Y_n)| - |V(X_n)|| = |V(G_n)|$ for all positive integers $n$, which means
\[
\lim_{n \to \infty} \frac{||V(Y_n)| - |V(X_n)||}{|V(X_n)|} = \lim_{n \to \infty} \frac{|V(G_n)|}{|V(X_n)|} = 0.
\]
Having shown that these limits are zero, there is an integer $N$ such that
\begin{align*}
\left|\int f~d\Psi(X_n) - \int f~d\Psi(Y_n)\right| &< \e + \frac{2L \cdot |\N_n|}{|V(X_n)|} + \frac{L \cdot ||V(Y_n)| - |V(X_n)||}{|V(X_n)|}\\
                                                   &< \e + 2L \cdot \e + L \cdot \e = (3L + 1) \cdot \e
\end{align*}
for all positive integers $n$ with $n \geq N$. Thus
\[
\lim_{n \to \infty} \left|\int f ~d\Psi(X_n) - \int f ~d\Psi(Y_n)\right| = 0,
\]
as required. In addition, assume that $(\Psi(X_n))$ converges weakly to $\mu$. Then
\[
\lim_{n \to \infty} \int f ~d\Psi(Y_n) = \lim_{n \to \infty} \int f ~d\Psi(X_n) = \int f ~d\mu
\]
for all $f \in \C(\Gr)$, and so $\mu$ is the weak limit of $(\Psi(Y_n))$. The proof of the converse is analogous.
\end{proof}

%
%
\subsection{Applications}

Let us discuss the following elementary example to put Theorem \ref{weak_limits_are_invariant_under_negligence} into perspective. Recall that $\ZZ_n$ is a cycle on $n$ vertices, and $P_{n-1}$ is a path on $(n - 1)$ vertices for all positive integers $n$ with $n \geq 3$. Consider the sequences $(\ZZ_n)_{n=3}^\infty$ and $(P_n)_{n=2}^\infty$. Clearly, $\ZZ_n \setminus \{z_n\} = P_{n-1}$ where $z_n \in V(\ZZ_n)$ for all positive integers $n$ with $n \geq 3$. Furthermore, the sequence $(z_n)_{n=3}^\infty$ is negligible in $(\ZZ_n)_{n=3}^\infty$. Since every cycle is vertex-transitive, we know that $(\Psi(\ZZ_n))_{n=3}^\infty = (\delta_{\ZZ_n})_{n=3}^\infty$, and it is easy to see that the weak limit of this sequence of Dirac measures is $\delta_\ZZ$. Then Theorem \ref{weak_limits_are_invariant_under_negligence} tells us that $(\Psi(P_n))_{n=2}^\infty$ converges weakly to $\delta_\ZZ$ too. Even though this fact is shown in this author's Honours project \cite{artemenko11a}, this proof is much less technical.

Another demonstration of the usefulness of Theorem \ref{weak_limits_are_invariant_under_negligence} is the following corollary. Denote by $\Lambda$ the infinite perfect binary tree, and let $\lambda$ be its first ancestor. Fix any vertex $t$ of $T$. Let $\Lambda_n = B_\Lambda(\lambda,n)$, and let $T_n = B_T(t,n)$ for all positive integers $n$. Intuitively, we expect that the weak limit of $(\Psi(\Lambda_n))$ is equal to that of $(\Psi(T_n))$ because $\Lambda$ may be thought of as $T$ with only two branches extruding from $t$.

\begin{cor}\label{branches_equiv_to_full_tree}
The weak limits of $(\Psi(T_n))$ and $(\Psi(\Lambda_n))$ coincide.
\end{cor}

\begin{proof}
Observe that $T_{n+1} \setminus \{t\} = 3\Lambda_n$ for all positive integers $n$. Since
\[
\lim_{n \to \infty} \frac{|V(\{t\})|}{|V(T_n)|} = 0
\]
and the weak limit of $(\Psi(T_n))$ is $\mu_S$, Theorem \ref{weak_limits_are_invariant_under_negligence} tells us that $(\Psi(\Lambda_n)) = (\Psi(3\Lambda_n))$ converges weakly to $\mu_S$ too.
\end{proof}

In addition, Theorem \ref{weak_limits_are_invariant_under_negligence} is easily specialized to better appeal to the examples presented at the beginning of this section. It is also useful in obtaining the weak limit of a sequence of complicated finite graphs by breaking them down into simpler parts.

\begin{cor}\label{weak_limits_are_invariant_under_negligence_corollary}
Let $(X_n)$ and $(G_n)$ be sequences of finite graphs where $X_n$ is connected, $(G_n)$ is negligible in $(X_n)$, and $Y_n = X_n \setminus G_n$ for all positive integers $n$. Fix the positive integers $\omega$ and $b_k$ for all $k \in \{1,2,\ldots,\omega\}$. Suppose that
\[
Y_n = \sum_{k=1}^\omega b_kX_n^k
\]
where $\{X_n^k ~:~ 1 \leq k \leq \omega\}$ is a set of pairwise distinct connected subgraphs of $X_n$ for all positive integers $n$. If the weak limit of $(\Psi(X_n^k))$ is $\mu^k$ for all $k \in \{1,2,\ldots,\omega\}$, then $(\Psi(X_n))$ converges weakly to
\[
\mu = \sum_{k=1}^\omega \left(\lim_{n \to \infty} \frac{b_k|V(X_n^k)|}{|V(Y_n)|}\right)\mu^k.
\]
\end{cor}

\begin{proof}
Let $f \in \C(\Gr)$. If $\mu^k$ is the weak limit of $(\Psi(X_n^k))$, then
\[
\lim_{n \to \infty} \int f ~d\Psi(X_n^k) = \int f ~d\mu^k
\]
for all $k \in \{1,2,\ldots,\omega\}$. By Corollary \ref{judiciality_of_disconnected_graphs_corollary},
\[
\Psi(Y_n) = \sum_{k=1}^\omega \left(\frac{b_k|V(X_n^k)|}{|V(Y_n)|}\right) \Psi(X_n^k),
\]
and so
\[
\int f~d\Psi(Y_n) = \sum_{k=1}^\omega \left(\frac{b_k|V(X_n^k)|}{|V(Y_n)|}\right) \cdot \int f ~d\Psi(X_n^k).
\]
Its limit is
\[
\lim_{n \to \infty} \int f ~d\Psi(Y_n) = \sum_{k=1}^\omega \left(\lim_{n \to \infty} \frac{b_k|V(X_n^k)|}{|V(Y_n)|}\right) \cdot \int f ~d\mu^k = \int f ~d\mu,
\]
and the result follows by Theorem \ref{weak_limits_are_invariant_under_negligence}.
\end{proof}

%% file: counterexamples_and_open_problems.tex
%
%
\section{Counterexamples and open problems}

Now that the reader has seen many true results, it is necessary to exhibit a few counterexamples to ensure us that not everything in this field of research is true. In fact, this section covers counterexamples for conjectures that, at first glance, could be seen as valid. The remainder deals with questions that this author has yet to answer, including a famous open problem posed by Schramm \cite{schramm08}, Aldous, and Lyons \cite{aldouslyons07}.

%
%
\subsection{Convergence of balls}

Let $X$ be an infinite connected graph, let $x \in V(X)$, and let $X_n = B_X(x,n)$ for all positive integers $n$. The question is whether or not $(\Psi(X_n))$ converges weakly. There are certainly examples for which this statement holds, such as the $3$-regular tree, but this is not true in general. To streamline the process of proving that the following counterexample is indeed valid, consider the following criterion.

\begin{prop}\label{counterexample_criterion}
Let $(X_n)$ be a sequence of graphs. The degree function $\deg : \Gr \to \RR$ is defined by $\deg[X,x] = \deg_X(x)$ for all $[X,x] \in \Gr$. If
\[
\lim_{n \to \infty} \int \deg ~d\Psi(X_n) = \lim_{n \to \infty} \frac{1}{|V(X_n)|} \sum_{x \in V(X_n)} \deg_{X_n}(x)
\]
does not exist, then $(\Psi(X_n))$ does not converge weakly.
\end{prop}

The intuition behind a counterexample $X$ is the following procedure: modify the subgraph of $\ZZ$ induced by the positive integers by attaching finite graphs of large degree at specific vertices along the path; this ensures that certain balls will have a ``jump'' in average degree.

It is possible to construct such a graph $X$ recursively. Consider the infinite path $\ZZ \cap [1,\infty)$, which is simply the subgraph of $\ZZ$ induced by the positive integers. Attach a copy of $K_6$, the complete graph on six vertices, at the vertex $k_1 = 1$, and let $l_1 = 2$. In general, attach a copy of $K_6$ at the vertices
\[
k_n,k_n + 1,\ldots,l_n - 1
\]
where
\[
\begin{cases}
\displaystyle{k_n = 15 \cdot \sum_{i=1}^{n-1} (l_i - k_i) - 2}\\
\displaystyle{l_n = \ceil{\frac{7}{6}k_n + \frac{1}{3}}}
\end{cases}
\]
for all positive integers $n \geq 2$. The resulting infinite graph $X$ is shown in Figure \ref{convergence_of_balls}.

\begin{figure}[ht]
\begin{center}
\includegraphics[scale=1.2]{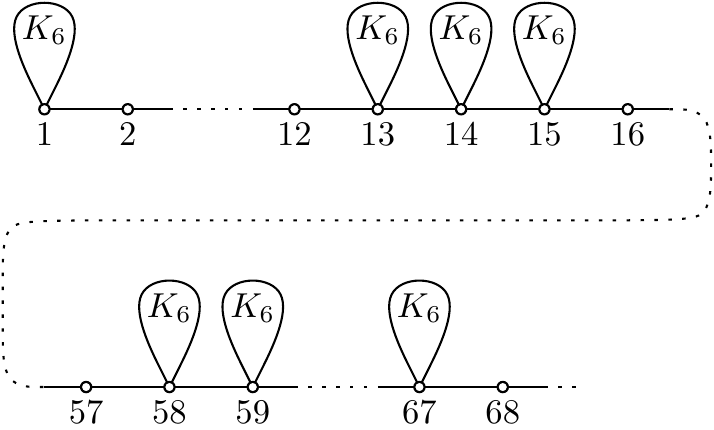}
\caption[The graph $X$ whose sequence of balls $(B_X(1,n))$ does not converge weakly]{The graph $X$ whose sequence of balls $(B_X(1,n))$ does not converge weakly. For clarity, $K_6$ is depicted by an inverted teardrop.}
\label{convergence_of_balls}
\end{center}
\end{figure}

It remains to prove that $(\Psi(X_n))$ does not converge weakly where $X_n = B_X(1,n)$ for all positive integers $n$. Let $n$ be a positive integer with $n \geq 2$. Note that
\[
\sum_{i=1}^{n-1} (l_i - k_i) = \frac{k_n + 2}{15}.
\]
For convenience, let $Y_n = X_{k_n-1}$. Since $Y_n$ is a path on $k_n$ vertices together with $\sum_{i=1}^{n-1} (l_i - k_i)$ copies of $K_6$, it follows that
\[
|V(Y_n)| = k_n + 5 \cdot \sum_{i=1}^{n-1} (l_i - k_i) = k_n + \frac{k_n + 2}{3} = \frac{4k_n + 2}{3}
\]
and
\[
\sum_{x \in V(Y_n)} \deg_{Y_n}(x) = 2(k_n - 1) + 30 \cdot \sum_{i=1}^{n-1} (l_i - k_i) = 2(k_n - 1) + 2(k_n + 2) = 4k_n + 2,
\]
which means the average degree is
\[
\frac{1}{|V(Y_n)|} \sum_{x \in V(Y_n)} \deg_{Y_n}(x) = 3.
\]
On the other hand, let $Z_n = X_{l_n-1}$. In this case, $Z_n$ is a path on $l_n$ vertices with $\sum_{i=1}^n (l_i - k_i)$ copies of $K_6$. Clearly,
\[
\sum_{i=1}^n (l_i - k_i) = (l_n - k_n) + \frac{k_n + 2}{15},
\]
which means
\[
|V(Z_n)| = l_n + 5 \cdot \sum_{i=1}^n (l_i - k_i) = l_n + 5(l_n - k_n) + \frac{k_n + 2}{3} = 6l_n - \frac{14}{3}k_n + \frac{2}{3}
\]
and
\begin{align*}
\sum_{x \in V(Z_n)} \deg_{Z_n}(x) &= 2(l_n - 1) + 30 \cdot \sum_{i=1}^n (l_i - k_i)\\
                                  &= 2(l_n - 1) + 30(l_n - k_n) + 2(k_n + 2)\\
                                  &= 32l_n - 28k_n + 2.
\end{align*}

Suppose that
\[
\frac{1}{|V(Z_n)|} \sum_{x \in V(Z_n)} \deg_{Z_n}(x) = \frac{32l_n - 28k_n + 2}{6l_n - (14/3)k_n + 2/3} < 4.
\]
Then
\[
32l_n - 28k_n + 2 < 24l_n - \frac{56}{3}k_n + \frac{8}{3},
\]
and so
\[
8l_n < \frac{28}{3}k_n + \frac{2}{3} ~~\Rightarrow~~ l_n < \frac{7}{6}k_n + \frac{1}{12} < \ceil{\frac{7}{6}k_n + \frac{1}{3}} = l_n;
\]
a contradiction. Thus
\[
\frac{1}{|V(Z_n)|} \sum_{x \in V(Z_n)} \deg_{Z_n}(x) \geq 4.
\]
It follows that $\lim_{n \to \infty} \int \deg ~d\Psi(Y_n) = 3$ and $\lim_{n \to \infty} \int \deg ~d\Psi(Z_n) \geq 4$, which means $\lim_{n \to \infty} \int \deg ~d\Psi(X_n)$ does not exist. By Proposition \ref{counterexample_criterion}, $(\Psi(X_n))$ does not converge weakly, as required.

%
%
\subsection{Weak limits of adjacent balls}

Let $X$ be an infinite connected graph, and let $x$ and $y$ be vertices of $X$. Let $X_n = B_X(x,n)$ and $Y_n = B_X(y,n)$ for all positive integers $n$. Suppose that $x$ and $y$ are adjacent. Intuitively, the reader may understandably expect a ball around $x$ to be ``similar'' to a ball around $y$ of the same radius. However, this is not the case, and it is possible to construct a graph $X$ such that the weak limits of $(\Psi(X_n))$ and $(\Psi(Y_n))$ are distinct.

\begin{figure}[ht]
\begin{center}
\includegraphics[scale=1.2]{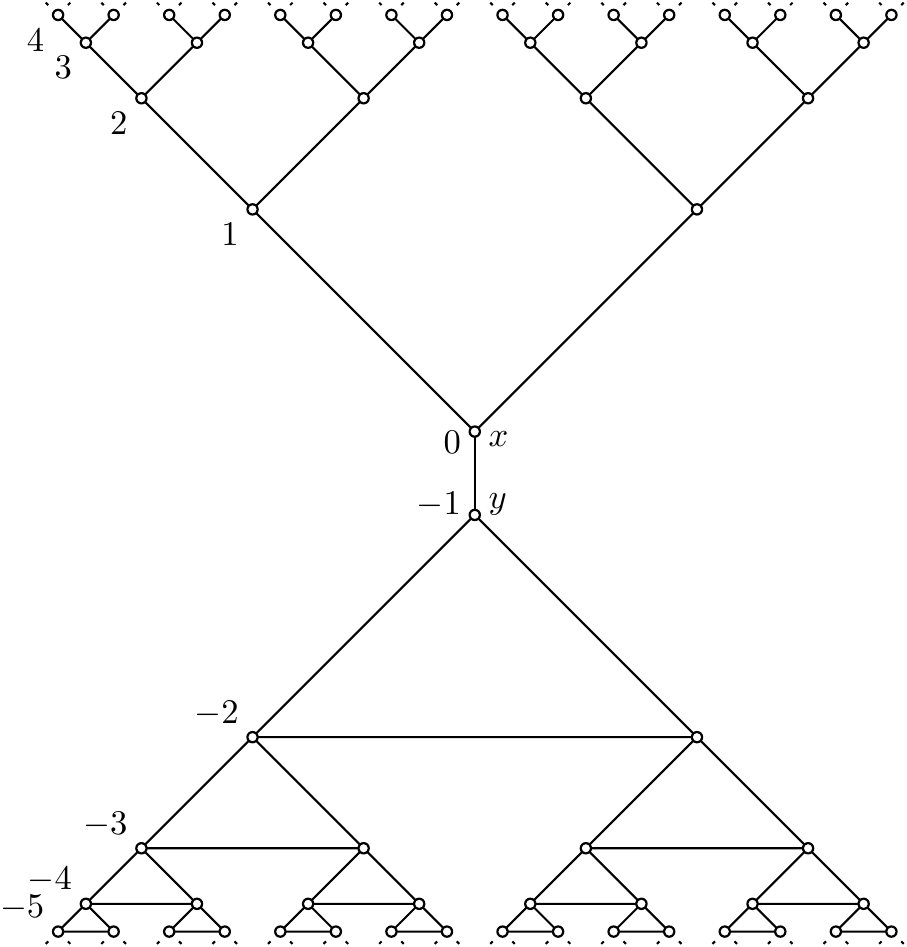}
\caption[The graph $X$, obtained by joining the infinite perfect binary tree and the barred binary tree]{The graph $X$, obtained by joining the infinite perfect binary tree and the barred binary tree. The orbits are labelled using the integers.}
\label{weak_limits_of_adjacent_balls}
\end{center}
\end{figure}

\begin{figure}[ht]
\begin{center}
\includegraphics[scale=1.2]{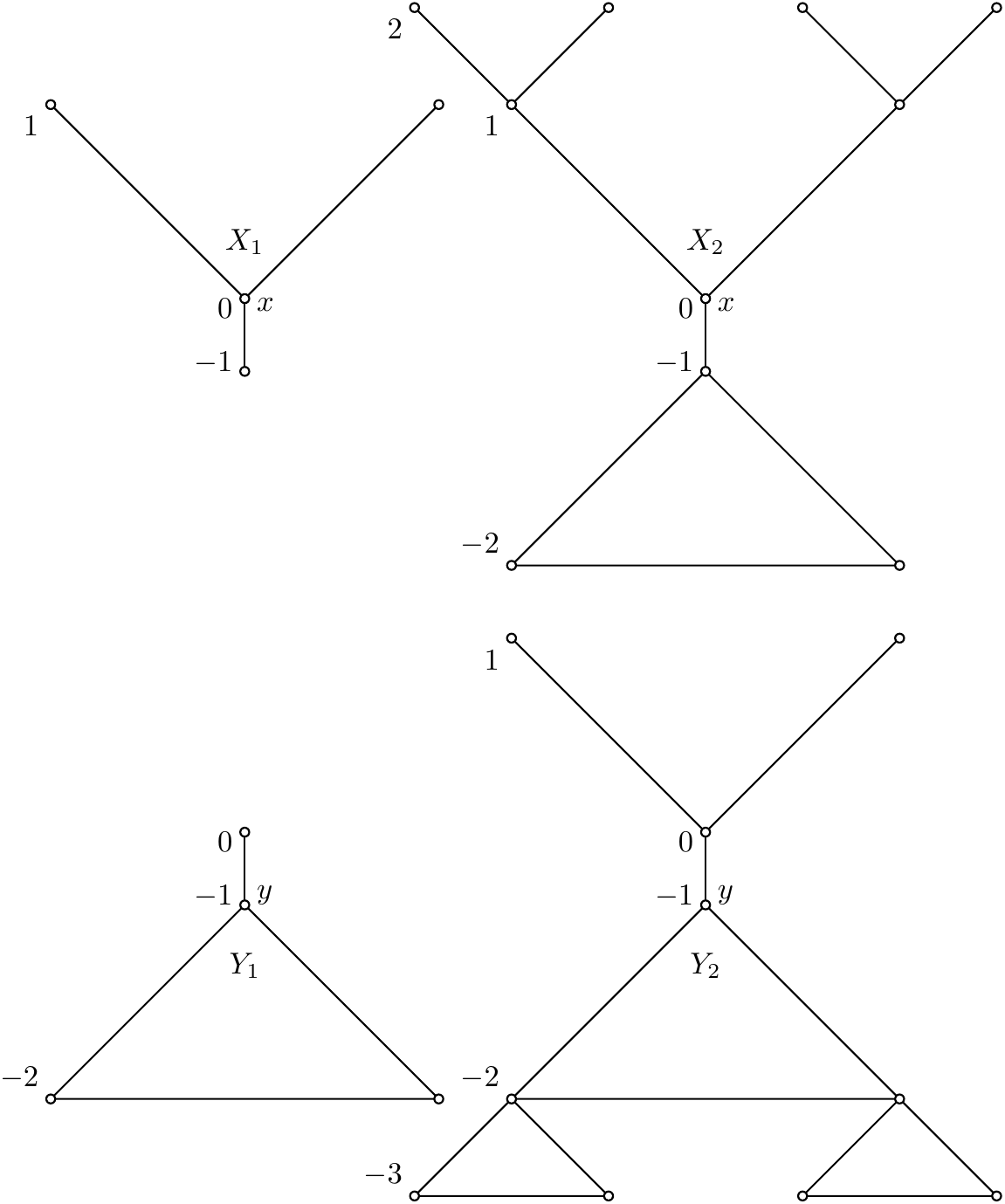}
\caption[The balls $X_1$, $X_2$, $Y_1$, and $Y_2$]{The balls $X_1 = B_X(x,1)$ and $X_2 = B_X(x,2)$ on the top, and $Y_1 = B_X(y,1)$ and $Y_2 = B_X(y,2)$ on the bottom.}
\label{balls_of_binary_tree_and_barred_binary_tree}
\end{center}
\end{figure}

To construct such an example, it is necessary to recall the infinite perfect binary tree $\Lambda$ and the barred binary tree $\bar\Lambda$. Let $x$ be the first ancestor of $\Lambda$, and let $y$ be the first ancestor of $\bar\Lambda$. The graph $X$ is defined as follows: $V(X) = V(\Lambda) \cup V(\bar\Lambda)$ and $E(X) = E(\Lambda) \cup E(\bar\Lambda) \cup \{\{x,y\}\}$. The result is shown in Figure \ref{weak_limits_of_adjacent_balls} with a labelling of its orbits. That is, $\Rcc(X) = \{[X,i] ~:~ i \in \ZZ\}$ where $[X,0] = [X,x]$ and $[X,-1] = [X,y]$. With this labelling, it is easy to describe the balls around $x$ and $y$. The balls of radius $1$ and $2$ are depicted in Figure \ref{balls_of_binary_tree_and_barred_binary_tree}. Indeed, $X_n = B_X(x,n)$ is induced by the union $\bigcup_{i=-n}^n \Aut(X)i$, and $Y_n = B_X(y,n)$ is induced by the union $\bigcup_{i=-(n+1)}^{n-1} \Aut(X)i$ for all positive integers $n$.

Next we use Theorem \ref{weak_limits_are_invariant_under_negligence} to more easily determine the weak limits of $(\Psi(X_n))$ and $(\Psi(Y_n))$. Let $X_n^1 = B_\Lambda(x,n - 1)$, $X_n^2 = B_{\bar\Lambda}(y,n - 1)$, $Y_n^1 = B_\Lambda(x,n - 2)$, and $Y_n^2 = B_{\bar\Lambda}(y,n)$ for all positive integers $n$ with $n \geq 2$. If $G_n = \{x\}$, then
\[
X_n \setminus G_n = 2X_n^1 + X_n^2
\]
and
\[
Y_n \setminus G_n = 2Y_n^1 + Y_n^2
\]
for all positive integers $n$ with $n \geq 2$.

By Theorem \ref{weak_limits_are_invariant_under_negligence}, it suffices to determine the weak limits of the sequences $(\Psi(X_n^1))$, $(\Psi(X_n^2))$, $(\Psi(Y_n^1))$, and $(\Psi(Y_n^2))$ because
\[
\lim_{n \to \infty} \frac{|V(G_n)|}{|V(X_n)|} = 0 = \lim_{n \to \infty} \frac{|V(G_n)|}{|V(Y_n)|}.
\]

Recall the graphs $S$ and $\bar{S}$, which were presented in the preliminaries. Proposition \ref{weak_limit_of_trees_barred_trees} and Corollary \ref{branches_equiv_to_full_tree} tell us that $(\Psi(X_n^1))$ and $(\Psi(X_n^2))$ converge weakly to $\mu^1 = \mu_S$ and $\mu^2 = \mu_{\bar{S}}$, respectively. Similarly, the weak limits of $(\Psi(Y_n^1))$ and $(\Psi(Y_n^2))$ are $\mu^1$ and $\mu^2$, respectively. Note that $|V(X_n)| = |V(Y_n)| = 3 \cdot 2^n - 2$, $|V(X_n^1)| = |V(X_n^2)| = 2^n - 1$, $|V(Y_n^1)| = 2^{n-1} - 1$, and $|V(Y_n^2)| = 2^{n+1} - 1$ for all positive integers $n$ with $n \geq 2$. By Theorem \ref{weak_limits_are_invariant_under_negligence}, $(\Psi(X_n))$ converges weakly to
\[
\left(\lim_{n \to \infty} \frac{2|V(X_n^1)|}{|V(X_n \setminus G_n)|}\right) \mu^1 + \left(\lim_{n \to \infty} \frac{|V(X_n^2)|}{|V(X_n \setminus G_n)|}\right) \mu^2 = \left(\frac{2}{3}\right) \mu^1 + \left(\frac{1}{3}\right) \mu^2
\]
and $(\Psi(Y_n))$ converges weakly to
\[
\left(\lim_{n \to \infty} \frac{2|V(Y_n^1)|}{|V(Y_n \setminus G_n)|}\right) \mu^1 + \left(\lim_{n \to \infty} \frac{|V(Y_n^2)|}{|V(Y_n \setminus G_n)|}\right) \mu^2 = \left(\frac{1}{3}\right) \mu^1 + \left(\frac{2}{3}\right) \mu^2,
\]
which are certainly distinct measures.

%
%
\subsection{Open problems}

Readers who have perused any other papers concerning the topics of unimodularity or weak limits of sequences of laws \cite{aldouslyons07,schramm08} are likely familiar with the following important, yet unanswered, question.

\begin{framed}
\begin{center}
``Is every unimodular measure the weak limit of a sequence of finite graphs?''
\end{center}
\end{framed}

The converse is true, and although the proof is omitted in this thesis, the reader is welcome to see this author's previous work \cite{artemenko11b} for a detailed argument.

Corollary \ref{unimodularity_for_vtransitive_graphs} tells us precisely when a connected vertex-transitive graph is judicial without relying on a specific measure. As previously mentioned, this author hopes for an extension of this result to more general judicial graphs.

Another question concerns the extreme points of $\U$, the set of unimodular measures. Proposition \ref{unimodular_sustained_by_connected_is_extreme_pt} tells us that a unimodular measure $\mu$ sustained by a connected graph is an extreme point. Whether or not every extreme point is a measure of this form is unknown.